\documentclass[11pt]{amsart}

\usepackage{amsfonts, amssymb, amscd}
\numberwithin{equation}{section}

\usepackage[symbol]{footmisc}

\usepackage{bm}
\usepackage{verbatim}
\usepackage{amssymb}
\usepackage{mathrsfs}
\usepackage{graphicx}
\usepackage[nospace,noadjust]{cite}
\usepackage{tikz-cd}
\usepackage{subcaption}
\usepackage{subfiles}
\usepackage[toc,page]{appendix}
\usepackage{mathtools}
\usepackage{comment}
\usepackage{enumerate}
\usepackage{enumitem}

\usepackage{graphicx}
\graphicspath{{images/}}

\usepackage{appendix}
\usepackage{hyperref}

\newcommand{\Cc}{\mathbb{C}}

\newcommand{\Pp}{\mathbb{P}}
\newcommand{\Qq}{\mathbb{Q}}

\newcommand{\Rr}{\mathbb{R}}

\newcommand{\Zz}{\mathbb{Z}}

\newcommand{\Span}{\operatorname{Span}}

\newcommand{\Center}{\operatorname{center}}

\newcommand{\mld}{{\rm{mld}}}

\newcommand{\lct}{\operatorname{lct}}

\newcommand{\Supp}{\operatorname{Supp}}

\newcommand{\mult}{\operatorname{mult}}

\newcommand{\lf}{\lfloor}
\newcommand{\rf}{\rfloor}
\newcommand{\xto}{\xrightarrow{f}}

\newcommand{\dto}{\dashrightarrow}

\newcommand{\Dd}{\mathcal{D}}

\newcommand{\Oo}{\mathcal{O}}
\newcommand{\Ii}{{\Gamma}}

\newcommand{\Ll}{\mathcal{L}}

\usepackage{todonotes}
\newtheorem*{thmN}{Theorem N}
\newtheorem*{thmP}{Theorem P}

\newtheorem{thm}{Theorem}[section]

\newtheorem{cor}[thm]{Corollary}
\newtheorem{lem}[thm]{Lemma}
\newtheorem{defn}[thm]{Definition}
\newtheorem{prop}[thm]{Proposition}

\newtheorem{claim}[thm]{Claim}
\theoremstyle{definition}
\newtheorem{rem}[thm]{Remark}

\theoremstyle{definition}

\begin{document}

\title{Boundedness of $n$-complements for generalized pairs}

%\date{\today}
%\setcounter{footnote}{-1}
%Boundedness of Index of Exceptional Singularities}
% Dedicated to Gang Tian on the occasion of this sixtieth birthday ’s Sixtieth Birthday with admirationyachesla

\author{Guodu Chen\\}

\address{Guodu Chen, Beijing International Center for Mathematical Research, Peking University, Beijing 100871, China}
\email{gdchen@pku.edu.cn}

\begin{abstract}
     We show the existence of $n$-complements for generalized pairs with additional Diophantine approximation properties when the coefficients of boundaries belong to a DCC set.
\end{abstract}
\date{\today}

\maketitle
\pagestyle{myheadings}\markboth{\hfill  Guodu Chen\hfill}{\hfill   Boundedness of $n$-complements for generalized pairs       \hfill}

% admitting an $\epsilon$-plt blow up
\tableofcontents

%\tableofcontentsExceptional Singularities 
\section{Introduction}\label{section1}

     We work over the field of complex numbers $\Cc.$ 
     
     The theory of complements was introduced by Shokurov when he proved the existence of log flips for threefolds \cite{Sho92}. It turns out that the theory of complements plays an important role in the recent development of birational geometry.
      The theory is further developed in \cite{Sho00,PS01,Bir04,PS09,Bir19,HLS19,Sho19,CH20}, see also \cite{FM18,Xuyanning19-1,FMX19}. In \cite{Bir19}, Birkar proved the boundedness of log canonical complements for Fano type varieties when the coefficients of boundaries belong to a hyperstandard sets $\Ii\subseteq [0,1]\cap \Qq$ which is a breakthrough in the study of Fano varieties. The boundedness of log canonical complements plays an important role in various contexts, see \cite{Bir16b,Bir18,LJH18,Bir19,Xu19,BLX19}.
     
     In \cite{HLS19}, Han, Liu and Shokurov proved the boundedness of log canonical complements with additional Diophantine approximation properties for Fano type varieties with any DCC set $\Ii\subseteq[0,1]$. The theory of complements in this case could also be applied to prove the ACC for minimal log discrepancies of exceptional singularities \cite{HLS19} and it has other applications.
     
     For the purpose of induction in the birational geometry, we need to study and explore the space of generalized pairs. They were first introduced in \cite{BZ16} to deal with the effectivity of Iitaka fibrations.
     %Roughly speaking, a generalized pair is a pair ``polarized'' by a nef divisor coming from a higher birational model, and one recovers the notation of a pair in the usual sense when the polarizing nef divisor becomes zero. 
     In recent years, they have found applications in various contexts, such as the boundedness of complements for Fano varieties \cite{Bir19}, and Fujita's spectrum conjecture \cite{HL17}, see \cite{Fil20,HL18,HL19,LT19} for more works.    
     
     It is not clear that one can address the nature questions for the usual pairs in the setting of generalized pairs, such as the cone theorem. On the other hand, it is known that the nonvanishing conjecture fails for genenrlaized polarized pairs, and we can only expect the numerical nonvanishing for genenrlaized polarized pairs, see \cite{HL20,LP18-1,LP18-2}. In this paper, we focus on the boundedness of complements for generalized pairs with additional Diophantine approximation properties, and show the existence of $n$-complements for generalized pairs with DCC coefficients. 
     
     Recall that the rational envelope $V\subseteq\Rr^m$ of a point $\bm{v}\in\Rr^m$ is the smallest affine subspace containing $\bm{v}$ which is defined over the rationals. For any point $\bm{v}:=(v_1,\dots,v_m)\in \Rr^m,$ we define $||\bm{v}||_{\infty}:=\max_{1\le i\le m}\{|v_i|\}$. Note that our definition of complements for generalized pairs is a slightly generalization of the usual definition of complements for generalized pairs, see Definition \ref{defn:complements}.

\begin{thm}\label{thm: existence of n complement for generalized pairs}
     Let $d,p$ and $s$ be positive integers, $\epsilon$ a positive real number, $\Ii\subseteq [0,1]$ a DCC set, $||.||$ a norm on $\Rr^s$, $\bm{v}_0\in\mathbb \Rr^s\setminus\Qq^s$ a point, $V\subseteq\Rr^s$ the rational envelope of $\bm{v}_0$, and $\bm{e}\in V$ a non-zero vector. Then there exists a positive integer $n$ and a point $\bm{v}\in V$ depending only on $d,p,s,\epsilon,\Ii,||.||,\bm{v}_0$ and $\bm{e}$ satisfying the following. 
	
	Assume that $(X,B+M)$ is a generalized pair with data $X'\xrightarrow{f}X\to Z$ and $M'=\sum \mu_jM_j'$ such that
	\begin{itemize}
		\item $\dim X=d$,	
		\item $X$ is of Fano type over $Z$,
		\item $B\in\Ii$, that is, the coefficients of $B$ belong to $\Ii,$ 
		\item $\mu_j\in\Gamma$ and $M_j'$ is b-Cartier nef$/Z$ for any $j$, and
		\item $(X/Z,B+M)$ is $\Rr$-complementary.
	\end{itemize}
	Then
	\begin{enumerate}
		\item $($existence of $n$-complements$)$ for any $z\in Z$, there exists an $n$-complement $(X/Z\ni z,B^++M^+)$ of $(X/Z\ni z,B+M)$, moreover, if $\Span_{\Qq_{\ge0}}(\bar{\Ii}\backslash\Qq)\cap (\Qq\backslash\{0\})=\emptyset$, then we may pick $B^+\ge B$ and $\mu_j^+\ge\mu_j$ for any $j$, where $M^{+'}=\sum \mu_j^+M_j'$ and $\bar{\Ii}$ is the closure of $\Ii$,
		\item $($divisibility$)$ $p|n$,
		\item $($rationality$)$ $n\bm{v}\in \Zz^s$,
		\item $($approximation$)$ $||\bm{v}_0-\bm{v}||<\frac{\epsilon}{n}$, and
		\item $($anisotropic$)$ $||\frac{\bm{v}_0-\bm{v}}{||\bm{v}_0-\bm{v}||}-\frac{\bm{e}}{||\bm{e}||}||<\epsilon.$
	\end{enumerate}	
\end{thm}
     When $M'=0$, $-(K_X+B)$ is nef over $Z$, and we remove the Diophantine approximation properties (2)--(5) in Theorem \ref{thm: existence of n complement for generalized pairs}, it was proved by Shokurov \cite{Sho00} for the case when $\dim X=2$ and $\Ii$ is the standard set, by Prokhorov and Shokurov \cite{PS09} for the case when $\dim X=3$ and $\Ii\subseteq\Qq$ is a hyperstandard set, by Birkar \cite[Theorem1.7, Theorem 1.8]{Bir19} for the case when $\Ii\subseteq\Qq$ is a hyperstandard set. When $M'=0$, it was proved by Han, Liu and Shokurov \cite{HLS19}. We refer readers to \cite{FM18} for the case when $\bar{\Ii}\subseteq[0,1]\cap\Qq$ without additional Diophantine approximation properties.
     
     \medskip

     In order to show Theorem \ref{thm: existence of n complement for generalized pairs}, we study a new class of complements, namely $(n,\Ii_0)$-decomposable $\Rr$-complements. Note that when $M'=0$, $(n,\Ii_0)$-decomposable $\Rr$-complements are the same as \cite[Definition 1.9]{HLS19}.

\begin{defn}
     Let $n$ be a positive integer, $\Ii_0\subseteq(0,1]$ a finite set and $(X,B+M)$ a generalized pair with data $X'\xto X\to Z$ and $M'=\sum \mu_jM_j'$. We say that $(X/Z\ni z,B^++M^+)$ is an \emph{$(n,\Ii_0)$-decomposable $\Rr$-complement} of $(X/Z\ni z,B+M)$ if
     \begin{enumerate}
        \item $(X/Z\ni z,B^++M^+)$ is an $\Rr$-complement of $(X/Z\ni z,B+M)$,
        \item $K_X+B^++M^+=\sum a_i(K_X+B_i^++M^+_i)$ for some boundaries $B_i^+$, nef parts $M^{+'}_i$ and $a_i\in\Ii_0$ with $\sum a_i=1$, and
        \item $(X/Z\ni z,B_i^++M^+_i)$ is an $n$-complement of itself for any $i$.
     \end{enumerate}
\end{defn}

     We will show the existence of $(n,\Ii_0)$-decomposable $\Rr$-complements for $\Rr$-complementary generalized pairs with DCC coefficients which is an important step in the proof of Theorem \ref{thm: existence of n complement for generalized pairs}.

\begin{thm}\label{thm: (n,I)complforfiniterat}
     	Let $d$ be a positive integer, and $\Ii\subseteq [0,1]$ a DCC set. Then there exists a positive integer $n$ and a finite set $\Ii_0\subseteq(0,1]$ depending only on $d$ and $\Ii$ satisfying the following. 
	
	Assume that $(X,B+M)$ is a generalized pair with data $X'\xto X\to Z$ and $M'=\sum \mu_jM_j$ such that
	\begin{itemize}
		\item $(X,B+M)$ is g-lc of dimension $d$,	
		\item $X$ is of Fano type over $Z$,
		\item $B\in\Ii,\mu_j\in\Gamma$ and $M_j'$ are b-Cartier nef$/Z$, and
		\item $(X/Z,B+M)$ is $\Rr$-complementary.
	\end{itemize}
	Then for any point $z\in Z$, there is an $(n,\Ii_0)$-decomposable $\Rr$-complement $(X/Z\ni z,B^++M^+)$ of $(X/Z\ni z,B+M)$. Moreover, if $\bar{\Ii}\subseteq\Qq$, then we may pick $\Ii_0=\{1\}$, and $(X/Z\ni z,B^++M^+)$ is a monotonic $n$-complement of $(X/Z\ni z,B+M)$.%\chen{moreover part remains}
	%Moreover, if $\Span_{\Qq_{\ge0}}(\bar{\Ii}\backslash\Qq)\cap (\Qq\backslash\{0\})=\emptyset$, then we may pick $B^+\ge B$.
\end{thm}

     As one of the main ingredients in the proof of Theorem \ref{thm: existence of n complement for generalized pairs}, we will show the existence of Nakamura's (generalized) lc rational polytopes.
     
\begin{thm}[Nakamura's (generalized) lc rational polytopes]\label{thm: uniformpolytopeforgenlc}
	Let $d,m$ and $l$ be positive integers, $\bm{v}_0:=(v_1^0,\ldots,v_{m+l}^0)\in\Rr^{m+l}$ a point and $V\subseteq\Rr^{m+l}$ the rational envelope of $\bm{v}_0$. Then there exists an open set $U\ni\bm{v}_0$ of $V$ depending only on $d,m$ and $\bm{v}_0$ satisfying the following.  
	
	Assume that $(X,(\sum_{i=1}^m v_i^0B_{i})+(\sum_{j=1}^{l}v_{m+j}^0M_j))$ is a generalized pair with data $X'\xto X\to Z$ and $\sum_{j=1}^{l}v_{m+j}^0M_j'$ such that
	\begin{enumerate} 
	   \item $(X/Z,(\sum_{i=1}^m v_i^0B_{i})+(\sum_{j=1}^{l}v_{m+j}^0M_i))$ is g-lc of dimension $d$,
	   \item $B_1,\ldots,B_m\ge0$ are Weil divisors on $X$, and
	   \item $M_j'$ are b-Cartier nef$/Z$ for any $1\le j\le l$.
	   %\item $(X/Z,(\sum_{i=1}^m v_i^0B_{i})+M)$ is $\Rr$-complementary.
	\end{enumerate}
	Then $(X,(\sum_{i=1}^m v_iB)+(\sum_{j=1}^{l}v_{m+j}M_j))$ is g-lc for any point $(v_1,\dots,v_{m+l})\in U$.
\end{thm}
     To show Theorem \ref{thm: uniformpolytopeforgenlc}, we generalized a result of Nakamura \cite[Theorem 1.6]{Nak16}, which is about perturbation of an irrational coefficient of g-lc pairs, see Theorem \ref{thm: uniformperturbationofglc}. The proof of Theorem \ref{thm: uniformperturbationofglc} is a combination of \cite{HMX14}, \cite{Nak16} and \cite{HLQ17}. It is worthwhile to point out that, we use Birkar-Borisov-Alexeev-Borisov Theorem to simplify the proof of Nakamura. Moreover, by using Theorem \ref{thm: uniformperturbationofglc}, we can show the existence of Han type polytopes for $\Rr$-complementary generalized pairs of Fano type over the base (See Theorem \ref{thm: uniformpolytopeforRcomp}). %see Theorem \ref{thm: uniformpolytopeforgenlc} and Theorem \ref{thm: uniformpolytopeforRcomp}.
     
\begin{comment}
     As a corollary of Theorem \ref{thm: existence of n complement for generalized pairs}, we show the existence of $n$-complements for generalized pairs when the coefficients of boundaries belong to a DCC set which is a generalized of \cite[Theorem 1.10]{Bir19}.
     
     \begin{cor}\label{cor: existence of n complement for generalized pairs}
	Let $d$ and $p$ be two positive integers, and $\Ii\subseteq [0,1]$ a DCC set. Then there exists a positive integer $p|n$ depending only on $d,p$ and $\Ii$ satisfying the following. 
	
	Assume that $(X,B+M)$ is a projective generalized pair with data $X'\xrightarrow{f}X$ and $M'$ such that
	\begin{itemize}
		\item $\dim X=d$,
		\item $X$ is of Fano type,
		\item $B\in\Ii$ and $rM'$ is b-Cartier, and
		\item $(X,B+M)$ is $\Rr$-complementary.
	\end{itemize}
	Then there exists an $n$-complements $(X,B^++M)$ of $(X,B+M)$. Moreover, if $\Span_{\Qq_{\ge0}}(\bar{\Ii}\backslash\Qq)\cap (\Qq\backslash\{0\})=\emptyset$, then we may pick $B^+\ge B$.
\end{cor}

\end{comment}

\medskip

\textit{Structure of the paper.}
     We outline the organization of the paper. In Section \ref{section2}, we introduce some notation and tools which will be used in this paper, and prove certain basic results. In Section \ref{section3}, we prove Theorem \ref{thm: uniformpolytopeforgenlc}. In Section \ref{section4'}, we show Theorem \ref{thm: relncomplforfiniterat} which is a generalized of \cite[Theorem 1.10]{Bir19}. In Section \ref{section4}, we prove Theorem \ref{thm: (n,I)complforfiniterat}. In Section \ref{section5}, we prove Theorem \ref{thm: existence of n complement for generalized pairs}.

\medskip

\textbf{Acknowledgements.} This work began when the author visited Jingjun Han at Johns Hopkins University in April of 2019, and the author would like to thank Jingjun Han for suggesting him the problem and also for useful discussions. Part of this work was done while the author visited the MIT Mathematics Department during 2018--2020 supported by China Scholarship Council (File No. 201806010039). The author would like to thank their hospitality. The author would like to thank his advisor Chenyang Xu for constant support and encouragement, and Vyacheslav V. Shokurov for teaching him the theory of complements. The author would also like to thank Jihao Liu for his comments.

\section{Preliminaries}\label{section2}
     In this section, we will collect some definitions and preliminary results which will be used in this paper.
     
\begin{comment}
\begin{defn}
     Let $\Ii\subseteq[0,1]$ be a set of real numbers, we define 
     $$\sum \Ii:=\{0\}\cup\{\sum_{p=1}^l i_p\mid i_p\in\Ii \text{ for }1\le p\le l,l\in\Zz_{>0}\}.$$  
\end{defn}    

\begin{defn}[DCC and ACC sets]\label{def: DCC and ACC}        
    We say that $\Ii\subseteq\Rr$ satisfies the \emph{descending chain condition} $($DCC$)$ if any decreasing sequence $a_{1} \ge a_{2} \ge \cdots$ in $\Ii$ stabilizes. We say that $\Ii$ satisfies the \emph{ascending chain condition} $($ACC$)$ if any increasing sequence in $\Ii$ stabilizes.
\end{defn}
We need the following well-known lemma.
\begin{lem}[{\cite[Proposition 3.4.1]{HMX14}}]\label{lemma DI}
     Let $\Ii\subseteq[0,1]$ be a set of real numbers, we have $D(D(\Ii))=D(\Ii)\cup\{1\}.$
\end{lem}
\end{comment}

\subsection{Arithmetic of sets}
Let $\Ii\subseteq[0,+\infty)$ be a set, we define 
     $$\sum \Ii:=\{0\}\cup\{\sum_{p=1}^l i_p\mid i_p\in\Ii \text{ for }1\le p\le l,l\in\Zz_{>0}\}.$$
     If $\Ii\subseteq[0,1]$, then we define
     $$\Phi(\Ii):=\{1-\frac{r}{m}\mid r\in\Ii,m\in\Zz_{>0}\}$$
     to be the set of \emph{hyperstandard multiplicities} associated to $\Ii$ (c.f. \cite[3.2]{PS09}). Note that if we add $1-r$ to $\Ii$ for any $r\in\Ii$, then we get $\Ii\subseteq\Phi(\Ii)$.
     
\begin{defn}[DCC and ACC sets]\label{def: DCC and ACC}        
    We say that $\Ii\subseteq\Rr$ satisfies the \emph{descending chain condition} (DCC) if any decreasing sequence $a_{1} \ge a_{2} \ge \cdots$ in $\Ii$ stabilizes. We say that $\Ii$ satisfies the \emph{ascending chain condition} (ACC) if any increasing sequence in $\Ii$ stabilizes.
\end{defn}

     Now assume that $\Ii\subseteq[0,1]\cap\Qq$ is a finite set. Then $\Phi(\Ii)$ is a DCC set of rational numbers whose only accumulation point is $1$.

\subsection{Divisors}
     We adopt the standard notation and definitions in \cite{BZ16} and \cite{Bir19}, and will freely use them.
     
     Let $X$ be a normal variety, $D:=\sum d_{i} D_{i}$ an $\Rr$-divisor and $a$ a real number. We define $\lfloor D\rfloor:=\sum \lfloor d_i\rfloor D_i, \{D\}:=\sum \{d_i\}D_i$, $D^{\ge a}:=\sum_{d_i\ge a}d_iD_i$, and $\lceil D\rceil:=\sum \lceil d_i\rceil D_i$.%and $||D||:=\max_i\{|d_i|\}$.
     
    %We say that $D$ is \emph{b-Cartier} if it is $\Qq$-Cartier and $\phi^*D$ is Cartier for some birational morphism $\phi:Y\to X$.

\begin{defn}[b-divisors]
     Let $X$ be a variety. A \emph{b-$\Rr$-Cartier b-divisor} over $X$ is the choice of a projective birational morphism $Y\to X$ from a normal variety $Y$ and an $\Rr$-Cartier divisor $M$ on $Y$ up to the following equivalence: another projective birational morphism $Y '\to  X$ from a normal variety $Y'$ and an $\Rr$-Cartier divisor $M'$ defines the same b-$\Rr$-Cartier b-divisor if there is a common resolution $W\to Y$ and $W\to Y'$ on which the pullbacks of $M$ and $M'$ coincide.
     
     A b-$\Rr$-Cartier b-divisor represented by some $Y\to X$ and $M$ is \emph{b-Cartier} if $M$ is b-Cartier, i.e., its pullback to some resolution is Cartier.
\end{defn}

\subsection{Generalized pairs}
\begin{defn}   
     We say $\pi: X \to Z$ is a \emph{contraction} if 
     $X$ and $Z$ are normal quasi-projective varieties, $\pi$ is a projective morphism, and $\pi_*\Oo_X = \Oo_Z$ $(\pi$ is not necessarily birational$)$. %In particular, $\pi$ is surjective and has connected fibers.
\end{defn}

\begin{defn}
     Let $X\to Z$ be a contraction. We say that $X$ is of \emph{Fano type} over $Z$ if $(X,B)$ is klt and $-(K_{X}+B)$ is big and nef over $Z$ for some boundary $B$.
\end{defn}

\begin{rem}
     Assume that $X$ is of Fano type over $Z$. Then we can run the MMP$/Z$ on any $\Rr$-Cartier $\Rr$-divisor $D$ on $X$ which terminates with some model $Y$ (c.f. \cite[Corollary 2.9]{PS09}). 
\end{rem}

\begin{defn}[generalized pairs] 
     A \emph{generalized pair} consists of 
     \begin{itemize}
          \item a normal variety $X$ equipped with a projective morphism $X\to Z$,
          \item an $\Rr$-divisor $B\ge0$ on $X$, and
          \item a b-$\Rr$-Cartier b-divisor over $X$ represented by some projective birational morphism $X'\xto X$ and $\Rr$-Cartier divisor $M'$ on $X'$, 
     \end{itemize}
     such that $M'$ is nef $/Z$, and $K_X+B+M$ is $\Rr$-Cartier, where $M:=f_{*}M'$. We call $B$ the \emph{boundary part} and $M$ the \emph{nef part}.
      In the rest of the paper, we may say that $M$ is the strict tranform of $M'$.
     
          We may say that $(X,B+M)$ is a generalized pair with data $X'\xto X\to Z$ and $M'$. 
          If $\dim Z=0$, the generalized pair is called \emph{projective}, and we will omit $Z$. If $Z=X$ and $X\to Z$ is the identity map, we will omit $Z$.
          Since a b-$\Rr$-Cartier b-divisor is defined birationally, we will often replace $X'$ by a higher model and replace $M'$ by its pullback.
          
          Possibly replacing $X$ by a higher model and $M'$ by its pullback, we may assume that $f$ is a log resolution of $(X,B)$, and write
          $$K_{X'}+B'+M'=f^*(K_X+B+M)$$
          for some uniquely determined $B'$. The \emph{generalized log discrepancy} of a divisor $E$ on $X'$ with respect to $(X,B+M)$ is $1-\mult_E B'$ and denoted by $a(E,X,B+M)$. We define the \emph{generalized minimal log discrepancy} of $(X,B+M)$ as
          $$\mld(X,B+M):=\min\{a(E,X,B+M)\mid E\text{ is a prime divisor over }X\}.$$
          
          We say that $(X,B+M)$ is \emph{generalized $\epsilon$-lc} $($respectively \emph{g-klt, g-lc}$)$ for some non-negative real number $\epsilon$ if $\mld(X,B+M)\ge \epsilon$ $($respectively $>0,\ge0)$. For a divisor $E$ over $X$ with $a(E,X,B+M)\le0$, we call $E$ a \emph{generalized nonklt place} and its image on $X$ a \emph{generalized nonklt center}. 
          
          A g-lc pair $(X,B+M)$ with data $X'\xto X\to Z$ and $M'$ is \emph{g-dlt} if $(X,B)$ is a dlt pair and if every generalized nonklt center of $(X,B+M)$ is a nonklt center of $(X,B)$. If in addition, the connected components of $\lf B\rf$ are irreducible, then we say that the generalized pair is \emph{generalized plt}.
\end{defn}

\begin{comment}
\begin{defn}[g-dlt]
     Let $(X,B+M)$ be a generalized pair with data $X'\xto X\to Z$ and $M'$. We say that $(X,B+M)$ is \emph{g-dlt} if it is g-lc and there is a closed subset $V\subseteq X$ such that
     \begin{enumerate}
        \item $X\setminus V$ is smooth and $B|_{X\setminus V}$ is a snc divisor, and
        \item if $a(E,X,B+M)=0$ for some prime divisor $E$ over $X$, then $\Center_X E\nsubseteq V$ and $\Center_XE\setminus V$ is a non-klt center of $(X,B)|_{X\setminus V}$.
     \end{enumerate}
\end{defn}
\begin{rem}
     If $(X,B+M)$ is a $\Qq$-factorial g-dlt pair with data $X'\xto X\to Z$ and $M'$, then $X$ is klt.
\end{rem}
\end{comment}

     %We note that although our definition of g-dlt is defferent from that of \cite[2.13(2)]{Bir19}, our definition of g-dlt is preserved under adjunctions and running MMPs \cite[Remark 2.4, Lemma 4.6]{HL18}.
We recall an adjunction formula for generalized pairs.

\begin{defn}[Generalized adjunction fomula]
     Let $(X,B+M)$ be a generalized pair with data $X'\xto X\to Z$ and $M'$. Assume that $S$ is the normalization of a component of $\lfloor B\rfloor$, and that $S'$ is its strict transform on $X'$. Possibly replacing $X'$ by a higher model, we may assume that $f$ is a log resolution of $(X,B)$ and write
     $$ K_{X'}+B'+M'=f^*(K_{X}+B+M). $$
     Then
     $$K_{S'}+B_{S'}+M_{S'}'=f^*(K_{X}+B+M)|_S,$$  
     where $B_{S'}:=(B'-S')|_{S'}$ and $M_{S'}':=M'|_{S'}$. Let $g:=f|_{S'}$ be the induced morphism, $B_{S}:=g_*B_{S'}$ and $M_{S}:=g_*M_{S'}'.$ Then we get 
     $$K_{S}+B_{S}+M_{S}=(K_X+B+M)|_S,$$
     and $(S,B_S+M_S)$ is a generalized pair with data $S'\stackrel{g}\to S$ and $M_{S'}'$, which is referred as the \emph{generalized adjunction formula}.
\end{defn}

\begin{defn}[g-lc thresholds]\label{defn: glcts}
     Let $(X,B+M)$ be a g-lc pair with data $X'\xto X\to Z$ and $M'$. Assume that $D\ge0$ is an $\Rr$-divisor on $X$ and $N'$ is a nef $/Z$ $\Rr$-divisor on $X'$, such that $D+N$ is $\Rr$-Cartier, where $N'=f_*N$. The \emph{g-lc threshold of $D+N$ with respect to $(X,B+M)$} is defined as
     $$\lct(X,B+M;D+N):=\sup\{t\mid(X,(B+tD)+(M+tN))\text{ is g-lc}\}.$$
\end{defn}

We also need the following results.

\begin{thm}[{ACC for g-lc thresholds, \cite[Theorem 1.5]{BZ16}}]\label{thm: ACCforglcts}
     Let $d$ be a positive integer and $\Ii$ a DCC set of non-negative real numbers. Then there exists an $ACC$ set $g\mathcal{LCT}(d,\Ii)$ depending only on $d$ and $\Ii$ satisfying the following.
     
     Assume that $(X, B+M),M',N'$ and $D$ are as in Definition \ref{defn: glcts} such that
     \begin{enumerate}
          \item $(X, B+M)$ is g-lc of dimension $d$,
          \item $B,D\in\Ii$,
          \item $M'=\sum \mu_jM_j'$ where $M'_j$ are nef$/Z$ Cartier divisors and $\mu_j\in \Ii$, and
          \item $N'=\sum \nu_kN_k'$ where $N_k'$ are nef$/Z$ Cartier divisors and $\nu_j\in \Ii$.
      \end{enumerate}
      Then $\lct(X,B+M;D+N)\in g\mathcal{LCT}(d,\Ii)$.
\end{thm}

\begin{thm}[{Global ACC for generalized pairs, \cite[Theorem 1.6]{BZ16}}]\label{thm: globalACC}
      Let $d$ be a positive integer and $\Ii$ a DCC set of non-negative real numbers. Then there exists a finite subset $\Ii_0 \subseteq \Ii$ depending only on $d$ and $\Ii$ satisfying the following.
      
       Assume that $(X, B+M)$ is a projective generalized pair with data $X'\xto X$ and $M'$ such that
      \begin{enumerate}
          \item  $(X, B+M)$ is g-lc of dimension $d$,
          \item $B\in\Ii$,
          \item  $M'=\sum\mu_jM_j'$ where $M_j'$ are nef Cartier divisors and $\mu_j\in\Ii$,
          \item  $\mu_j=0$ if $M_j'\equiv 0,$ and
          \item  $K_{X}+B+M\equiv 0.$
     \end{enumerate}
     Then $B\in\Ii_0$ and $\mu_j\in\Ii_0$ for any $j$.
\end{thm}

\subsection{MMP for generalized pairs} For generalized pairs, one can ask whether one can run MMP and whether it terminates. However the MMP for generalized pairs is not completed established, but some important cases could be derived from the standard MMP. We elaborate these results which are developed in \cite[\S 4]{BZ16}.

\begin{comment}     
     Let $(X,B+M)$ be a $\Qq$-factorial generalized pair with data $X'\xto X\to Z$ and $M'$, and $A$ a general ample$/Z$ divisor on $X$. Moreover, assume that
     
     $ (\star)$  for any $0<\epsilon\ll1$, there exists a boundary $\Delta_{\epsilon}\sim_{\Rr,Z}B+M'+\epsilon A$, such that $(X,\Delta_{\epsilon})$ is klt.
     
     Under assumption $(\star)$, we can run a g-MMP$/Z$ on $(K_X+B+M)$ with scaling of $A$, although the termination is not known (c.f. \cite[\S 4]{BZ16}).
     
     The following lemma shows that assumption $(\star)$ is satisfied in two cases.    
\begin{lem}[{\cite[Lemma 3.5]{HL18}}]
     Let $(X,B+M)$ be a g-lc pair with data $X'\xto X\to Z$ and $M'$, and $A$ an ample$/Z$ divisor, such that either
     \begin{enumerate}
       \item $(X,B+M)$ is g-klt, or
       \item $(X,C)$ is klt for some boundary $C$.
    \end{enumerate}   
    Then there exists a boundary $\Delta\sim_{\Rr,Z}B+M+A$ such that $(X,\Delta)$ is klt. Moreover, if $X$ is $\Qq$-factorial, we may run a g-MMP$/Z$ on $(K_X+B+M)$.
\end{lem}
\end{comment}
     
     \begin{lem}[g-dlt modification {\cite[Proposition 3.9]{HL18}}]
     Let $(X,B+M)$ be a g-lc pair with data $X'\xto X\to Z$ and $M'$. Then possibly replacing $X'$ by a higher model, there exist a $\Qq$-factorial g-dlt pair $(Y,B_Y+M_Y)$ with data $X'\stackrel{g}\to Y\to Z$ and $M'$, and a contraction $\phi:Y\to X$ such that $K_Y+B_Y+M_Y=\phi^*(K_X+B+M)$. Moreover, each exceptional divisor of $\phi$ is a component of $\lf B_Y\rf$. 
\end{lem}

     We call $(Y,B_Y+M_Y)$ a g-dlt modification of $(X,B+M)$.

     We may use the following lemma frequently without citing it in this paper.
     
\begin{lem}[{\cite[Lemma 4.4]{BZ16}, \cite[Lemma 3.5]{HL18}}]
     Let $(X,B+M)$ be a $\Qq$-factorial g-lc pair with data $X'\xto X\to Z$ and $M'$ such that $K_{X}+B+M$ is not pseudo-effective$/Z$, and either
     \begin{enumerate}
       \item $(X,B+M)$ is g-klt, or
       \item $(X,C)$ is klt for some boundary $C$.
    \end{enumerate}
    Then any g-MMP$/Z$ on $K_{X}+B+M$ with scaling of some ample$/Z$ $\Rr$-Cartier $\Rr$-divisor terminates with a Mori fiber space.
\end{lem}

\subsection{Complements}
We now introduce complements for generalized pairs. Note that our definition of complements is a slightly generalization of the usual definition of complements for generalized pairs (c.f. \cite{Bir19}).
\begin{defn}[Complements]\label{defn:complements}
     Let $(X,B+M)$ be a generalized pair with data $X'\xto X\to Z$ and $M'=\sum \mu_jM_j'$ such that $\mu_j\ge0$, $M_j'$ are b-Cartier nef$/Z$ divisors, and $\mu_j=0$ iff $M_j$ is trivial over $Z$. We say that $(X/Z\ni z,B^++M^+)$ is an \emph{$\Rr$-complement} of $(X/Z\ni z,B+M)$ if $(X,B^++M^+)$ is g-lc, $B^+\ge B,\mu_j^+\ge\mu_j$ for any $j$, and $K_X+B^++M^+\equiv0$ over a neighborhood of $z$, where $M^{+'}=\sum_j\mu_j^+M_j'$.
     
     Let $n$ be a positive integer. %Let $(X,B+M)$ be a generalized pair with data $X'\xto X\to Z$ and $M'$ such that $M'=\sum \mu_jM_j'$, where $\mu_j\ge0$, $M_j'$ are $\Qq$-Cartier nef$/Z$ divisors and $\mu_j=0$ iff $M_j'\equiv0$ over $Z$. 
     We say that a generalized pair $(X/Z\ni z,B^++M^+)$ is an \emph{$n$-complement} of $(X/Z\ni z,B+M)$, if over a neighborhood of $z$, we have
     \begin{enumerate}
        \item $(X,B^++M^+)$ is generalized lc,
        \item $n(K_X+B^++M^+)\sim 0$,
        \item $nB^+\ge n\lfloor B\rfloor+\lf(n+1)\{B\}\rf$, and
        \item $n\mu_j^+\ge n\lfloor \mu_j\rfloor+\lf(n+1)\{\mu_j\}\rf$ and $n(M^+)'$ is b-Cartier, where $M^{+'}=\sum \mu_j^+M_j'$.
     \end{enumerate}
\end{defn}
     We say that $(X/Z\ni z,B^++M^+)$ is a \emph{monotonic $n$-complement} of $(X/Z\ni z,B+M)$ if we additionally have $B^+\ge B$ and $\mu_j^+\ge\mu_j$ for any $j$. We say that $(X/Z\ni z,B+M)$ is \emph{$\Rr$-complementary} (respectively \emph{$n$-complementary}) if it has an $\Rr$-complement (respectively $n$-complement).
     
     If $\dim Z=0$, we will omit $Z$ and $z$.
     If for any  $z\in Z$, $(X/Z\ni z,B+M)$ is $\Rr$-complementary (respectively $n$-complementary), then we say that $(X/Z,B+M)$ is $\Rr$-complementary (respectively $n$-complementary) .
     
The following lemma is well-known to experts (c.f. \cite[6.1]{Bir19}). We will use the lemma frequently without citing it in this paper. 
\begin{lem}\label{lem: pullbackcomplements}
     Let $(X,B+M)$ be a generalized pair with data $X'\xto X\to Z$ and $M'$, and $z\in Z$ a point. Assume that $g:X\dashrightarrow X''/Z$ is a birational contraction and $B'',M''$ are the strict transforms of $B,M$ respectively.
     \begin{enumerate}
        \item If $(X/Z\ni z,B+M)$ is $\Rr$-complementary, then $(X''/Z\ni z,B''+M'')$ is $\Rr$-complementary.
        \item Let $n$ be a positive integer. If $g$ is $-(K_X+B+M)$-non-positive and $(X''/Z\ni z,B''+M'')$ is $\Rr$-complementary (respectively monotonic $n$-complementary), then $(X/Z\ni z,B+M)$ is $\Rr$-complementary (respectively monotonic $n$-complementary).
     \end{enumerate}
\end{lem}
     
\begin{comment}
     We need the following result of Birkar. Indeed we will generalized the following theorem in Section \ref{section4'}.
\begin{thm}[{\cite[Theorem 1.10]{Bir16b}}]\label{thm: ncomplforfiniterat}
     	Let $d$ and $r$ be positive integers and $\Ii\subseteq [0,1]\cap\Qq$ a finite set. Then there exists a positive integer $n$ depending only on $d$ and $\Ii$ satisfying the following. 
	
	Assume that $(X,B+M)$ is a projective generalized pair with data $X'\xto X$ and $M'$ such that
	\begin{itemize}
		\item $(X,B+M)$ is g-lc of dimension $d$,	
		\item $X$ is Fano type, 
		\item $B\in\Ii$ and $rM'$ is b-Cartier, and
		
		\item $-(K_X+B+M)$ is nef.
	\end{itemize}
	Then $(X,B+M)$ has a monotonic $n$-complement.
\end{thm}
\end{comment}
\section{Uniform rational polytopes}\label{section3}

\subsection{Accumulation points of g-lc thresholds}
     The goal of this subsection is to prove Theorem N which is a  generalization of \cite[Proposition 3.10]{Nak16}. 
     %Let $\Ii\subseteq[0,+\infty)$ be a set, we define $$\sum \Ii:=\{0\}\cup\{\sum_{p=1}^l i_p\mid i_p\in\Ii \text{ for }1\le p\le l,l\in\Zz_{>0}\}.$$
     
     Let $X$ be a variety, $B_i$ distinct prime divisors on $X$ and $d_i(t):\Rr\to \Rr$ $\Rr$-linear functions. We call the formal finite sum $\sum_i d_i(t)B_i$ a \emph{linear functional divisor}.    
    
\begin{defn}[$\Dd_c(\Ii),\Dd_c^{nef}(\Ii)$]
    Let $c$ be a non-negative real number and $0\in \Ii\subseteq[0,+\infty)$ a set. 
    
    $(1)$ We define $\Dd_c(\Ii)$ to be the set of linear functional divisors $B(t):=\sum_i{d_i(t)B_i},$ such that
     \begin{itemize}
          \item for each $i$, either $d_i(t)=1$ or $d_i(t)=\frac{m-1+u+kt}{m}$, where $m\in \Zz_{>0},$ $u\in \sum\Ii$ and $k\in \Zz$, and
          \item $u+kt$ above can be written as $u+kt=\sum_j{(u_j+k_jt)},$ where $u_j\in \Ii,k_j\in \Zz$, and $u_j+k_jc\ge 0$ for each $j$.
     \end{itemize}
     For convenience, we may write $B(t)\in \Dd_c(\Ii)$.  
     
    $(2)$ We define $\Dd_c^{nef}(\Ii)$ to be the set of $X'\xto X\to Z$ of varieties together with linear functional divisors $M'(t):=\sum_j{\mu_j(t)M_j'},$ and $M(t)=f_*M(t)'$, such that
     \begin{itemize}
       \item $f$ is a projective birational morphism and $X\to Z$ is a contraction,
       \item $M_j'$ is b-Cartier nef$/Z$ divisor on $X'$ for each $j$, and 
       \item for each $j$, either $\mu_j(t)=1$ or $\mu_j(t)=v+nt=\sum_i (v_i+n_it)$, where $v_i\in \Ii,n_i\in \Zz,v_i+n_ic\ge 0$ for each $i$.%and $\mu_j(c)<1$.
     \end{itemize}
     For convenience, we may write $(X'\xto X\to Z,M(t))\in\Dd_c^{nef}(\Ii)$. If $\dim Z=0$, then we may omit $Z$.
\end{defn}
     For convenience, if $B(t)\in\Dd_c(\Ii)$, and $(X'\xto X\to Z,M(t))\in\Dd_c^{nef}(\Ii)$ such that $(X,B(c)+M(c))$ is generalized pair with data $X'\xto X\to Z$ and $M'(c)$, then we say $(X,B(t)+M(t))$ is a generalized pair with data $X'\xto X\to Z$ and $M'(t)$.
     
     The form of the coefficient $d_i(t)$ is preserved by generalized adjunction. The proof is similar to \cite[Lemma 3.2]{Nak16} and \cite[Lemma 3.3]{Bir19}, we may omit it.
\begin{lem}
     Let $c$ be a non-negative real number and $0\in \Ii\subseteq[0,+\infty)$ a set. Suppose that $(X,B(t)+M(t))$ is a $\Qq$-factorial generalized pair with data $X'\xto X\to Z$ and $M'(t)$ such that
  \begin{enumerate}
    \item $(X,B(c)+M(c))$ is g-lc,
    \item $B(t)=\sum_i d_i(t)B_i\in \Dd_c(\Ii),$
    \item $d_0(t)=1,$ and $d_i(c)>0$ for each $i,$ and
    \item $(X'\xto X\to Z,M(t))\in\Dd_c^{nef}(\Ii)$.
  \end{enumerate}
  Let $S$ be the normalization of $B_0,$ and
  $$K_{S}+B_{S}(t)+M_{S}(t)=(K_{X}+B(t)+M(t))|_{S}$$
  the generalized adjunction. Then $B_{S}(t)\in \Dd_c(\Ii).$
\end{lem}

\begin{comment}
\begin{proof}
     The proof is similar to \cite[Lemma 3.2]{Nak16} and \cite[Lemma 3.3]{Bir19}. Also see \cite[Lemma 3.15]{HLQ17}.
\end{proof}
\end{comment}

     We define $\mathfrak{L}_d(\Ii)$, the set of g-lc thresholds derived from a set of non-negative real numbers $\Ii$, and $\mathfrak{N}_d(\Ii)$, the set of generalized numerical trivial thresholds derived from $\Ii.$

\begin{defn}[$\mathfrak{L}_d(\Ii)$]
     Let $d$ be a positive integer, and $0\in\Ii\subseteq[0,+\infty)$ a set. We write $c\in \mathfrak{L}_d(\Ii)$ if there exists a $\Qq$-factorial generalized pair $(X,B(t)+M(t))$ with data $X'\xto X\to Z$ and $M'(t)$ such that
    \begin{enumerate}    
       \item $(X,B(c)+M(t))$ is a g-lc pair of dimension $\le d,$
       \item $B(t)\in \Dd_c(\Ii),$ 
       \item $(X'\xto X\to Z,M(t))\in\Dd_c^{nef}(\Ii)$, and
       \item either $(X,B(c+\epsilon)+M(c+\epsilon))$ is not g-lc for any $\epsilon>0$, or $(X,B(c-\epsilon)+M(c-\epsilon))$ is not g-lc for any $\epsilon>0$.
    \end{enumerate}
\end{defn}

\begin{defn}[$\mathfrak{N}_d(\Ii)$]
     Let $d$ be a positive integer, and $0\in\Ii\subseteq[0,+\infty)$ a set. We write $c\in \mathfrak{N}_d(\Ii)$ if there exists a $\Qq$-factorial projective generalized pair $(X,B(t)+M(t))$ with data $X'\xto X$ and $M'(t)$ such that 
     \begin{enumerate}    
        \item $(X,B(c)+M(c))$ is a g-lc pair of dimension $\le d,$
        \item $B(t)\in \Dd_c(\Ii),$
        \item $(X'\xto X\to Z,M(t))\in\Dd_c^{nef}(\Ii)$, and
        \item $K_X+B(c)+M(c)\equiv 0,$ and
        \item $K_X+B(c')+M(c')\not\equiv 0$ for some $c'$.
     \end{enumerate}
\end{defn}

\begin{lem}
     Let $d\ge2$ be a positive integer, and $0\in \Ii\subseteq[0,+\infty)$ a set. Then $\mathfrak{L}_{d}(\Ii)\subseteq\mathfrak{N}_{d-1}(\Ii).$
\end{lem}

\begin{proof}
     The proof is similar to \cite[Theorem 3.6]{Nak16}.
\end{proof}

\begin{thm}\label{thm:accumulationpoints1}
     Let $d$ be a positive integer, and $\Ii\subseteq[0,1]$ a finite set. Then the accumulation points of $\mathfrak{N}_d(\Ii)$ belong to $\Span_{\Qq}(\Ii\cup\{1\})$. In particular, the accumulation points of $\mathfrak{L}_{d}(\Ii)$ belong to $\Span_{\Qq}(\Ii\cup\{1\})$.
\end{thm}

\begin{comment}
     As a corollary, we can prove the rationality of the accumulation points of $\mathfrak{L}_d(\Ii)$.
     
\begin{cor}\label{cor:accumulationpoints}
     Let $d$ be a positive integer, and $\Ii\subseteq[0,1]$ a finite set. Then the accumulation points of $\mathfrak{L}_{d}(\Ii)$ belong to $\Span_{\Qq}(\Ii\cup\{1\})$.
\end{cor}
\end{comment}

     Theorem \ref{thm:accumulationpoints1} immediately follows from the following theorem.

\begin{thmN}\label{thmN}
     Let $d$ be a positive integer, $c$ a non-negative real number and $0\in \Ii\subseteq[0,+\infty)$ a finite set.
         Suppose that for each $i\in \Zz_{>0},$ there exist positive real numbers $c_i$ and $\Qq$-factorial projective generalized pairs $(X_i,(A_i+B_i(t))+M_i(t))$ with data $X_i'\stackrel{f_i}\to X_i$ and $M_i'(t)$, such that
     \begin{enumerate}    
        \item $(X_i,(A_i+B_i(c_i))+M_i(c_i))$ is a g-lc pair of dimension $\le d,$
        \item the coefficients of $A_i$ are approaching $1$,
        \item $B_i(t)\in \Dd_{c_i}(\Ii)$,
        \item $(X_i'\xto X_i,M_i(t))\in\Dd_{c_i}^{nef}(\Ii)$,
        \item $\lim c_i=c,$   
        \item $K_{X_i}+A_i+B_i(c_i)+M_i(c_i)\equiv 0,$ and
        \item $K_{X_i}+A_i+B_i(c_i')+M_i(c_i')\not\equiv 0$ for some $c_i'$.
     \end{enumerate}
     Then $c\in\Span_{\Qq}(\Ii\cup\{1\}).$
\end{thmN}

\begin{thmP}\label{thmP}
     Let $d$ be a positive integer, $c$ a non-negative real number and $0\in \Ii\subseteq[0,+\infty)$ a finite set.     
     Suppose that for each $i\in \Zz_{>0},$ there exists positive real numbers $c_i$ and $\Qq$-factorial projective generalized pairs $(X_i,(A_i+B_i(t))+M_i(t))$ with data $X_i'\stackrel{f_i}\to X_i$ and $M_i'(t)$, such that
     \begin{enumerate}    
        \item $(X_i,(A_i+B_i(c_i))+M_i)$ is a g-klt pair of dimension $\le d,$
        \item $X_i$ is a Fano variety with Picard number $1$, 
        \item the coefficients of $A_i$ are approaching $1$,
        \item $B_i(t)\in \Dd_{c_i}(\Ii)$, 
        \item $(X_i'\xto X_i,M_i(t))\in\Dd_{c_i}^{nef}(\Ii)$,
        \item $\lim c_i=c,$   
        \item $K_{X_i}+A_i+B_i(c_i)+M_i(c_i)\equiv 0,$ and
        \item $K_{X_i}+A_i+B_i(c_i')+M_i(c_i)\not\equiv 0$ for some $c_i'$.
     \end{enumerate}
     Then $c\in\Span_{\Qq}(\Ii\cup\{1\}).$
\end{thmP}

     We will prove Theorem N and Theorem P inductively. Before that, we put some additional conditions.
\begin{rem}\label{rem: contralcoeff}
     In Theorem N and Theorem P, we may write $B_i(t):=\sum_l{d_{il}(t)B_{il}}$ and $M_i(t)':=\sum_k\mu_{ik}(t)M_{ik}'$ by definition. Possibly replacing $A_{i}$ and $B_{i}(t)$, we may assume that $d_{ij}(t)$ is not identically one. By \cite[Lemma 3.7]{Nak16}, we may assume that 
     $$\Ii\cap c\Zz_{>0}=\emptyset.$$ 
     By \cite[Claim 3.13]{Nak16}, possibly passing to a subsequence, we may assume that
      \begin{itemize}
          %\item  if $d_{il}(t)=\frac{m_{il}-1+u_{il}+k_{il}t}{m_{il}}$ and $\mu_{ik}(t)=\sum (v_{ik}+n_{ik}t)$, where $m_{il}\in \Zz_{>0},u_{il},v_{ik}\in \sum\Ii$ and $k_{il},n_{ik}\in \Zz$, then $u_{il},v_{ik},k_{il},n_{ik}$ have only finitely many possibilities,
          \item  if $d_{il}(t)=\frac{m_{il}-1+u_{il}+k_{il}t}{m_{il}}$, where $m_{il}\in \Zz_{>0},u_{il}\in \sum\Ii$ and $k_{il}\in \Zz$, then $u_{il},k_{il}$ have only finitely many possibilities,
          \item $\epsilon_0<d_{il}(c_i)<1$ for some positive real number $\epsilon_0$ and any $i,l,$
          \item $d_{il}(c)>0$ for any $i,l$, and
          \item the set $\{d_{il}(c),\mu_{ik}(c)\mid i,l,k\}$ satisfy the DCC.
      \end{itemize}
\end{rem}

\begin{lem}\label{rem2}
     Notation as in Theorem N. Then possibly passing to a subsequence, we may assume that $(X_i,(\lceil A_i\rceil +B_i(c))+M_i(c))$ is g-lc.
\end{lem}

\begin{proof}
     We may assume that $c_i$ are decreasing to $c$. Write $B_i(t)=B_{i0}+t(B_{i}^{+}-B_{i}^-)$ and $M_i(t)=M_{i0}+t(M_{i}^{+}-M_{i}^-)$. Note that $(X_i,(A_i+B_{i0}+cB_{i}^+-c_iB_{i}^-)+(M{i0}+cM_{i}^+-c_iM_{i}^-))$ is g-lc. As the coefficients of $B_{i0}+cB_{i}^+-c_iB_{i}^-$ and $M{i0}+cM_{i}^+-c_iM_{i}^-$ belong to a DCC set, and the coefficients of $A_i$ are increasing, by Theorem \ref{thm: dcc limit glc divisor}, possibly passing to a subsequence, we may assume that $(X_i,(\lceil A_i\rceil +B_i(c))+M_i(c))$ is g-lc for any $i.$
\end{proof}

     In the following, by ``(Theorem  N)$_{d}$'' (respectively ``(Theorem  P)$_{d}$''), we mean Theorem N (respectively Theorem P) with dimension $\le d$.
     
\begin{prop}\label{prop1}
     Let $d\ge2$ be an integer. Then $($Theorem N$)_{d-1}$ implies $($Theorem N$)_{d}$ if there exists a component $S_i$ of $\lf A_i\rf$ such that $(K_{X_i}+A_i+B_i(c_i')+M_i(c_i'))|_{S_i}\not\equiv 0$ for any $i$.
\end{prop}

\begin{proof}
     The proposition follows by generalized adjunction.
\end{proof}

\begin{prop}\label{prop2}
     Let $d\ge2$ be an integer. Then $($Theorem N$)_{d-1}$ implies $($Theorem N$)_{d}$ if there exists a Mori fiber space $X_{i}\to Z_{i}$ with $\dim Z_{i}>0$ and $\Supp A_{i}$ dominates $Z_{i}$ for any $i$.
\end{prop}

\begin{proof}
     Let $F_i$ be a general fiber of $g_i:X_{i}\to Z_{i}$. Then restricting to $F_i$ gives 
     $$K_{F_i}+A_{F_i}+B_{F_i}(t)+M_{F_i}(t):=(K_{X_i}+A_i+B_i(t)+M_i(t))|_{F_i},$$
     and $(F_i, (A_{F_i}+B_{F_i}(c_i))+M_{F_i}(c_i))$ is a g-lc pair of dimension $\le d-1$.
     %where $F_i'$ is the strict transform of $F_i$ on $X_i'$ and $M_{F_i'}'=M_i'|_{F_i'}.$ 
     By Lemma \ref{rem2}, we may assume that $(F_i, (A_{F_i}+B_{F_i}(c))+M_{F_i}(c))$ is g-lc.
     
     If $K_{F_i}+A_{F_i}+B_{F_i}(c_{i}')+M_{F_i}(c_{i}')\not\equiv0$, then we are done by assumption. Suppose that $K_{F_i}+A_{F_i}+B_{F_i}(c_{i}')+M_{F_i}(c_{i}')\equiv0$, then $K_{F_i}+A_{F_i}+B_{F_i}(c)+M_{F_i}(c)\equiv0$. By Theorem \ref{thm: globalACC}, possibly passing to a subsequence, we have $\lf A_i\rf=A_i$ for any $i$. In particular, there exists a component $S_i$ of $A_i$ such that $g_i(S_i)=Z_i$. It follows that $(K_{X_i}+A_i+B_i(c)+M_i(c))|_{S_i}\not\equiv0$. Hence we are done by Proposition \ref{prop1}.
\end{proof}
     
\begin{prop}\label{prop:pandnimplyn}
     Let $d\ge2$ be an integer. Then $($Theorem P$)_{d}$ and $($Theorem N$)_{d-1}$ imply $($Theorem N$)_{d}$. 
\end{prop}

\begin{proof}
     Possibly passing to a subsequence, we may assume that $\dim X_i=d$.      
     Suppose that $(X_i,A_i+B_i(c_i)+M_i(c_i))$ is not g-klt for any $i$. Possibly replacing $(X_i,(A_i+B_i(c_i))+M_i(c_i))$ by a g-dlt model, we may assume that $(X_i,(A_i+B_i(c_i))+M_i(c_i))$ is $\Qq$-factorial g-dlt, and $\lf A_i\rf=\lf A_i+B_i(c_i)\rf\neq0$. We may run a g-MMP on $(K_{X_i}+A_i+B_i(c_i)+M_i(c_i)-\lf A_i\rf)$ with scaling of an ample divisor,
     $$X_i:=X_i^{(0)}\dto X_i^{(1)}\dto\cdots\dto X_i^{(m_i)}\to Z_i$$
     which terminates with a Mori fiber space $X_i^{(m_i)}\to Z_i$, as $K_{X_i}+A_i+B_i(c_i)+M_i(c_i)-\lf A_i\rf\equiv -\lf A_i\rf$ is not pseudo-effective. Let $A_i^{(j)},B_i^{(j)}(t)$ and $M_i^{(j)}(t)$ be the strict transforms of $A_i,B_i(t)$ and $M_i(t)$ on $X_i^{(j)}$ respectively, and $D_i^{(j)}:=K_{X_i^{(j)}}+A_i^{(j)}+B_i^{(j)}(c_i')+M_i^{(j)}(c_i')$ for any $1\le j\le m_i$. 
     
     Assume that there exists an integer $1\le l_i\le m_i-1$, such that $D_i^{(j)}\not\equiv0$ for any $1\le j\le l_i$, and $D_i^{(l_i+1)}\equiv0.$ Then $f_i^{(l_i)}:X_i^{(l_i)}\to X_i^{(l_i+1)}$ is a divisorial contraction. Let $E_i$ be the $f_i^{(l_i)}$-exceptional divisor. Since $D_i^{(l_i)}\not\equiv0,$ $D_i^{(l_i)}-(f_i^{(l_i)})^*D_i^{(l_i)+1}=\alpha_i E_i$ for some non-zero real number $\alpha_i$. Since $f_i^{(l_i)}$ is $\lf A_i^{(l_i)}\rf$-positive, there exists a component $S_i$ of $\Supp \lf A_i^{(l_i)}\rf$ such that $E_i|_{S_i}\not\equiv0$ and we are done by Proposition \ref{prop1}.
     
     We may assume that $D_i^{(m_i)}\not\equiv0$. Possibly replacing $X_i$ by ${X}_i^{(m_i)}$, and $A_i,B_i(t),M_i(t)$ by its strict transforms respectively, we may assume that $X_i$ admits a Mori fiber space $g_i:X_i\to Z_i$. Note that $\Supp A_i$ dominates $Z_i$ as $g_i$ is $\lf A_i\rf$-positive. If $\dim Z_i>0$, we are done by Proposition \ref{prop2}. 
     If $\dim Z_i=0$, then we are done since then $X_i$ has Picard number one and then $(K_{X_i}+A_i+B_i(c_i')+M_i(c_i'))|_{T_i}\not\equiv0$ for any component $T_i$ of $\Supp \lf A_i\rf$.

     \medskip

     Now assume that $(X_i,(A_i+B_i(c_i))+M_i(c_i))$ is g-klt for any $i$. It follows that $({X_i},A_i+B_i(c_i'')+M_i(c_i''))$ is g-klt and $K_{X_i}+A_i+B_i(c_i'')+M_i(c_i'')$ is not pseudo-effective for some positive real number $c_i''$. We may run a g-MMP on $(K_{X_i}+A_i+B_i(c_i'')+M_i(c_i''))$ with scaling of an ample divisor which terminates with a Mori fiber space $h_i:\tilde{X}_i\to V_i$. Since every step of the g-MMP is $(K_{X_i}+A_i+B_i(c_i'')+M_i(c_i''))$-negative and $K_{X_i}+A_i+B_i(c_i'')+M_i(c_i'')$ is not pseudo-effective, $(K_{X_i}+A_i+B_i(c_i'')+M_i(c_i''))\not\equiv0$. 
     We may replacing $X_i$ by $\tilde{X}_i$, and $A_i,B_i(t),M_i(c_i'')$ by its strict transforms respectively, and therefore assume that $X_i$ admits a Mori fiber space $h_i:X_i\to V_i$.
     
     If $\dim V_i>0$, let $F_k$ be a general fiber of $h_i$. 
     As $h_i$ is $(K_{X_i}+A_i+B_i(c_i'')+M_i(c_i''))$-negative, 
     $(K_{X_i}+A_i+B_i(c_{i}'')+M_i(c_i''))|_{F_i}\not\equiv0.$ 
     Then we are done by restricting to $F_{i}$. If $\dim V_i=0$, then $X_i$ has Picard number one and the statement holds.
\end{proof}

\begin{prop}\label{prop: pifepsilonlc}
     Let $\epsilon$ be a positive real number and $d$ a positive integer. Then $($Theorem P$)_d$ holds if we additionally assume that $X_{i}$ is $\epsilon$-lc for every $i$. In particular, $($Theorem N$)_1$ and $($Theorem P$)_1$ hold. 
\end{prop}

\begin{proof}
     Possibly passing to a subsequence, we may assume that and $\dim X_i=d$ and $A_{i}=0$. By \cite[Theorem 1.1]{Bir16b}, $X_i$ belongs to a bounded family. In particular, there is a very ample divisor $H_i$ on $X_i$, such that $H_i^d$ and $-K_{X_i}\cdot H_i^{d-1}$ are bounded from above. %Since $X_{i}$ has Picard number one, $M_{i}$ is nef and $rM$ is integral. 
     By the assumption, we have
     $$(K_{X_i}+B_i(c_i)+M_i(c_i))\cdot H_i^{d-1}=0$$
     and $(B(c_i)+M(c_i))\cdot H_i^{d-1}\neq( B(c_i')+M(c_i'))\cdot H_i^{d-1}$
     which imply that $c_i=c$ for $i$ sufficiently large by Remark \ref{rem: contralcoeff}.
         
    If $d=1$, then $X_i$ is $\Pp^1$ for any $i$. Thus $($Theorem N$)_1$ and $($Theorem P$)_1$ hold.
\end{proof}

\begin{prop} \label{prop: nimpliesp}
     Let $d\ge2$ be an integer. Then $($Theorem N$)_{d-1}$ implies $($Theorem P$)_d$.
\end{prop}

\begin{proof}
     We may assume that $c_i$ is decreasing and $\dim X_i=d$. Let 
     $$a_i:=\mld(X_i,A_i+B_i(c_i)+M_i(c_i))=a(E_i,X_i,A_i+B_i(c_i)+M_i(c_i))$$ 
     for some prime divisor $E_i$ over $X_i$. By Proposition \ref{prop: pifepsilonlc}, we may assume that $1>a_i$ is decreasing, and $\lim_{i\to+\infty} a_i=0$.
     
     We first reduce to the case when $A_{i}\neq0$. Possibly replacing $X_i'$ by a higher model, there exists a morphism $\phi_i:\tilde{X}_i\to X_i$ extracting $E_i$. Write 
     $$K_{\tilde{X}_i}+(1-a_i)E_i+\tilde{A}_i+\tilde{B}_i(c_i)+\tilde{M}_i(c_i)=\phi_i^*(K_{X_i}+A_i+B_i(c_i)+M_i(c_i)),$$ 
     where $\tilde{A}_i,\tilde{B}_i(t)$ and $\tilde{M}_i(t)$ are the strict transforms of $A_i,B_i(t)$ and $M_i'(t)$ on $\tilde{X}_i$ respectively. Possibly replacing $X_i,A_i,B_i(t)$ and $M_i(t)$ by $\tilde{X}_i,(1-a_i)E_i+\tilde{A}_i,\tilde{B}_i(t)$ and $\tilde{M}_i(t)$ respectively, we may assume that $A_i\neq 0$. Note that $X_{i}$ may has Picard number $\ge2$.
     
     Since $K_{X_{i}}+B_{i}(c_{i})+M_{i}(c_i)\equiv -A_{i}$ is not pseudo-effective, we may run a g-MMP on $(K_{X_{i}}+B_{i}(c_{i})+M_{i}(c_i))$ which terminates with a Mori fiber space ${X}_{i}''\to Z_{i}$. Let $A_i'',B_i(t)''$ and $M_i''(t)$ be the strict transforms of $A_i,B_i(t)$ and $M_i(t)$ respectively. Since each step of the g-MMP is $A_i$-positive, $A_i''\neq0$. We claim that $K_{X_{i}}+A_i+B_{i}(c_{i}')+M_{i}(c_i')\not\equiv0$. Otherwise $K_{X_{i}}+A_i+B_{i}(c)+M_{i}(c)\equiv0$, which contradicts Theorem \ref{thm: globalACC} as $({X_i},(A_i+B_i(c))+M_i(c))$ is g-lc (by Lemma \ref{rem2}) and the coefficients of $A_i$ are approaching $1$. 
     
     We may replace $X_{i}$ by ${X}_{i}''$, and assume that $X_{i}$ admits a Mori fiber space $g_{i}:X_{i}\to Z_{i}$. Note that we have $A_{i}\neq0$ and $\Supp A_{i}$ dominates $Z_{i}$ since $g_{i}$ is $A_{i}$-positive. If $\dim Z_{i}>0$, then we are done by Proposition \ref{prop2}. Hence we may assume that $\dim Z_{i}=0$ and $X_{i}$ has Picard number one.
     
     \begin{claim}\label{lem roundupglc}
          Possibly passing to a subsequence, we may assume that $K_{X_i}+A_i+B_i(c)+M_i(c)$ is not ample for any $i$.
     \end{claim}
     \begin{proof}[Proof of Claim \ref{lem roundupglc}]
     By Lemma \ref{rem2}, possibly passing to a subsequence, we may assume that $({X_i},(A_i+B_i(c))+M_i(c))$ is g-lc for any $i$. We may write $B_i(t):=B_{i0}+t(B_{i}^{+}-B_{i}^-)$ and $M_i(t):=M_{i0}+t(M_{i}^{+}-M_{i}^-)$, where $B_{i}^+\ge 0$ and $B_{i}^-\ge 0$ have no common components.      
     %\chen{Here something ooops}
     Suppose on the contrary that $K_{X_i}+A_i+B_i(c)+M_i(c)$ is ample. Then $(c-c_i)(B_{i}^{+}-B_{i}^-+M_{i}^{+}-M_{i}^-)$ is ample, since $K_{X_i}+A_i+B_i(c_i)+M_i(c_i)\equiv 0.$ In particular, $B_{i}^{+}-B_{i}^-+M_{i}^{+}-M_{i}^-$ is antiample, as $c_i>c$ by assumption. Then $B_{i}^-+M_i^-\equiv \gamma_iB_{i}^{+}+M_i^+$ for some real number $\gamma_i>1.$ Let $t_i:=c_i-\frac{c_i-c}{\gamma_i}.$ Then $c<t_i<c_i$ and $(X_i,A_i+B_{i0}+cB_{i}^+-t_iB_{i}^-+cM_{i}^+-t_iM_{i}^-)$ is g-lc. Moreover, we have
       $$K_{X_i}+A_i+B_{i0}+cB_{i}^+-t_iB_{i}^-+cM_{i}^+-t_iM_{i}^-\equiv 0,$$
       which contradicts Theorem \ref{thm: globalACC}.
     \end{proof}

     Now we can finish the proof. If $(X_i,(\lceil A_i\rceil +B_i(c_i))+M_i(c_i))$ is not g-lc, then we set
     $$e_i:=\sup\{t\mid(X_i,(\lceil A_i\rceil +B_i(t))+M_i(c_i)) \text{ is g-lc}\}.$$
     Then $e_i\in \Ll_d(r,\Ii)\subseteq \mathfrak{N}_{d-1}(r,\Ii),$ and $\lim e_i=\lim c_i=c$, we are done. Thus we may assume that $(X_i,(\lceil A_i\rceil +B_i(c_i))+M_i(c_i))$ is g-lc. 
     
     We define $\alpha_i$ and $ \beta_i$ as
     $$K_{X_i}+\lceil A_i\rceil +B_i(\alpha_i)+M_i(\alpha_i)\equiv 0, K_{X_i}+\beta_i\lceil A_i\rceil +B_i(c)+M_i(c)\equiv 0.$$
     We have that $\alpha_i<c_i,$ by Claim \ref{lem roundupglc} and the assumption that $K_{X_i}+A_i+B_i(c_i)+M_i(c_i)\equiv 0$. If $\alpha_i< c<c_i.$ Then $K_{X_i}+\lceil A_i\rceil +B_i(c)+M_i(c)$ is ample and thus $\beta_i<1.$ As the coefficients of $A_i$ are approaching one, $\lim \beta_i=1$. In particular, we may assume that the set of coefficients of $\beta_i\lceil A_i\rceil +B_i(c)+M_i(c)$ satisfies the DCC, which contradicts Theorem \ref{thm: globalACC}. Thus $c\le\alpha_i<c_i,$ and $\lim \alpha_i=c.$ Note that then $(X_i,(\lceil A_i\rceil+B_i(\alpha_i))+M_i(\alpha_i))$ is g-lc and not g-klt, as both $(X_i,(\lceil A_i\rceil+B_i(c))+M_i(c))$ and $(X_i,(\lceil A_i\rceil+B_i(c_i))+M_i(c_i))$ are g-lc. We may replace $A_i$ by $\lceil A_i\rceil$, and we are done by Proposition \ref{prop1}.
\end{proof}

\begin{proof}[Proof of Theorem N and Theorem P]
     The theorems follow from Proposition \ref{prop1}, Proposition \ref{prop2}, Proposition \ref{prop:pandnimplyn}, Proposition \ref{prop: pifepsilonlc} and Proposition \ref{prop: nimpliesp}.
\end{proof}

\subsection{Nakamura's (generalized) lc rational polytopes}
     In this subsection, we show Theorem \ref{thm: uniformpolytopeforgenlc}.
     %and study the properties of Nakamura type polytopes and Han polytopes. 
     
     The proof of Theorem \ref{thm: uniformperturbationofglc} is similar to that of \cite[Theorem 1.6]{Nak16}. Note that for generalized pairs $(X,B(t)+M(t))$ satisfying the conditions of Theorem \ref{thm: uniformperturbationofglc}, $K_{X}+B(t)+M(t)$ is $\Rr$-Cartier for any $t\in\Rr$, by \cite[Lemma 5.4]{HLS19}.
     
\begin{thm}\label{thm: uniformperturbationofglc}
     Let $d,c,m$ and $l$ be positive integers, $r_0:=1,r_1,\dots,r_{c}$ real numbers which are linearly independent over $\Qq$, and $s_1,\ldots, s_{m},\dots,s_{m+l}: \Rr^{c+1}\to\Rr$ $\Qq$-linear functions. Then there exists a positive real number $\epsilon$ depending only on $d_1,\dots, r_{c}$ and $s_1,\dots, s_{m+l}$ satisfying the following. 
     
     Assume that $(X,B(t)+M(t))$ is a generalized pair with data $X'\xto X\to Z$ and $M'(t):=\sum_{j=1}^ls_{m+j}(r_0,\ldots,r_{c-1},t)M_j'$ such that
     \begin{enumerate}
        \item $(X,B(r_c)+M(r_c))$ is g-lc of dimension $d$,
        \item $B(t):=\sum_{i=1}^ms_i(r_0,\ldots,r_{c-1},t)B_i$, where $B_i$ are Weil divisors on $X$, and
        \item $M_j'$ is b-Cartier nef$/Z$ for any $1\le j\le l$.
     \end{enumerate}
     Then $(X,B(t)+M(t))$ is g-lc for any $t$ satisfying $|t-r_{c}|\le\epsilon$.
\end{thm}
     
\begin{proof}
     We may write $s_i(x_0,\dots,x_c):=\sum_{j=0}^{c}q_{ij}x_j,$ for any $1\le i\le m+l$, where $q_{ij}$ is a rational number for any $i,j$. Let $n$ be a positive integer such that $nq_{ic}\in\Zz$ for any $1\le i\le m+l$. Since $s_i(r_0,\dots,r_c)\ge0$ and $r_0,\dots,r_c$ are linearly independent over $\Qq$, there exist two rational numbers $t^-,t^+$ such that $t^-<r_c<t^+,$ and $s_i(r_0,\dots,r_{c-1},t)\ge0$ for any $t$ satisfying $t^-\le t\le t^+$.
     
     Suppose on the contrary that there exist generalized pairs $(X_i$, $B_i(t)+M_i(t))$ with data $X_i'\xrightarrow{f_i} X_i\to Z_{i}$ and $M'_{i}(t)$ satisfying the conditions and that either $\lim_{i\to\infty}  h_i^+=r_c$ or $\lim_{i\to\infty}  h_i^-=r_c$, where $h_i^+$ and $h_i^-$ are defined as
     $$h_i^+:=\sup\{t\ge r_c\mid (X_i,B_i(t)+M_i(t))\text{ is g-lc}\},$$
     $$h_i^-:=\sup\{t\le r_c\mid (X_i,B_i(t)+M_i(t))\text{ is g-lc}\}.$$
     
     For each $i$, possibly replacing $X'_{i}$ by a higher model, there exists a g-dlt modification $g_{i}:Y_{i}\to X_{i}$ of $(X_{i},B_{i}(r_{c})+M_{i}(r_c))$, such that  
     $$K_{Y_{i}}+B_{Y_{i}}(r_c)+E_{i}+M_{Y_{i}}(r_c)=g_{i}^{*}(K_{X_{i}}+B_{i}(r_{c})+M_{i}(r_c)),$$
     where $M_{Y_{i}}(t),B_{Y_{i}}(t)$ are the strict transforms of $M'_{i}(t),B_{i}(t)$ on $Y_{i}$ respectively, and $E_{i}$ is the sum of reduced $g_{i}$-exceptional divisors.
     
     Since $r_0,\dots,r_c$ are linearly independent over $\Qq$, we have
     $$K_{Y_{i}}+B_{Y_{i}}(t)+E_{i}+M_{Y_{i}}(t)=g_{i}^{*}(K_{X_{i}}+B_{i}(t)+M_{i}(t))$$
     for any $t\in\Rr$. 
     Possibly replacing $(X_{i},B_{i}(t)+M_{i}(t))$ by $(Y_{i},(B_{Y_{i}}(t)+E_{i})+M_{Y_{i}}(t))$, we may assume that $X_{i}$ is $\Qq$-factorial for any $i$.
     
     By our assumltion,  we may write $B_i(t):=\sum_{j=1}^m s_j(r_0,\ldots,r_{c-1},t)B_{ij}$ and $M'_{i}(t):=\sum_{k=1}^ls_{m+k}(r_0,\ldots,r_{c-1},t)M_{ik}'$. Without loss of generality, we may assume that $\lim_{i\to\infty} h_i^-=r_c$ and $t^-\le h_i^-\le r_c$. Let 
     $$\Ii:=\{s_i(r_0,\dots,r_{c-1},t^-)\mid 1\le i\le m\}$$
     be a finite set. Since
     $$B_i(t)=B_i(t^-)+\sum_{j=1}^m \frac{t-t^-}{n}(nq_{jc})B_{ij},$$
     and
     $$M_i(t)=M_i(t^-)+\sum_{k=1}^l \frac{t-t^-}{n}(nq_{kc})M_{ik},$$
     we have that $\frac{h_i^--t^-}{n}\in\mathfrak{L}_{d}(\Ii)$. Hence 
     $$\frac{r_c-t^-}{n}\in\Span_{\Qq}(\{\Ii\cup\{1\}\})\subseteq\Span_{\Qq}(\{r_0,\dots,r_{c-1}\})$$
     by Theorem\ref{thm:accumulationpoints1}, a contradiction.
\end{proof}

Following the same arguments as in \cite[Theorem 5.6]{HLS19}, we have the following result. %We may call such polytopes Nakamura type polytopes.

\begin{proof}[Proof of Theorem \ref{thm: uniformpolytopeforgenlc}]
The result follows from Theorem \ref{thm: uniformperturbationofglc}.
\end{proof}

\begin{rem}
     When $\bm{v}\in\Qq^m$, then $U=V=\{\bm{v}\}$. We also note that $U$ dose not depend on $X$.
\end{rem}

%\subsection{Uniform perturbation of $\Rr$-complementary generalized pairs}

\subsection{Han type polytopes for $\Rr$-complementary (generalized) pairs}

\begin{thm}\label{thm: uniformperturbationofRcomp}
     Let $d,c,m$ and $l$ be positive integers, $r_0:=1,r_1,\dots,r_{c}$ real numbers which are linearly independent over $\Qq$, and $s_1,\ldots, s_{m+l}: \Rr^{c+1}\to\Rr$ $\Qq$-linear functions. Then there exists a positive real number $\epsilon$ depending only on $d_1,\dots, r_{c}$ and $s_1,\dots, s_{m+l}$ satisfying the following. 
     
     Assume that $(X,B(t)+M(t))$ is a generalized pair with data $X'\xto X\to Z$ and $M'(t):=\sum_{j=1}^ls_{m+j}(r_0,\ldots,r_{c-1},t)M_j'$ such that
     \begin{enumerate}
        \item $\dim X=d$, 
        \item $X$ is of Fano type over $Z$,
        \item $B(t):=\sum_{i=1}^ms_i(r_0,\ldots,r_{c-1},t)B_i$, where $B_i$ are Weil divisors on $X$, 
        \item $M_j'$ is b-Cartier nef$/Z$ for any $1\le j\le l$, and
        \item $(X/Z,B(r_c)+M(r_c))$ is $\Rr$-complementary.
     \end{enumerate}
     Then $(X/Z,B(t)+M(t))$ is $\Rr$-complementary for any $t$ satisfying $|t-r_{c}|\le\epsilon$.
\end{thm}
     The proof is very similar to that of \cite[Theorem 5.16]{HLS19}.

\begin{proof}
     We claim that there exists a positive real number $\epsilon$, such that $-(K_{X}+B(t)+M(t))$ is pseudo-effective for any $t$ satisfying $|t-r_{c}|\le\epsilon$. Suppose that the claim does not hold. By Theorem \ref{thm: uniformperturbationofglc}, there exist generalized pairs $(X_l,B_l(t)+M_l(t))$ with data $X_l'\xrightarrow{f_l} X_l\to Z_{l}$ and $M'_{l}(t)$ satisfying the conditions and either $\lim h_l^+=r_c$ or $\lim h_l^-=r_c$, where
     $$h_l^+:=\sup\{t\ge r_c\mid -(K_{X_l}+B_{l}(t)+M_l(t))\text{ is pseudo-effective}/Z\},$$
     $$h_l^-:=\sup\{t\le r_c\mid -(K_{X_l}+B_{l}(t)+M_l(t))\text{ is pseudo-effective}/Z\}.$$
	
	Without loss of generality, we may assume that $\lim h_l^{+}=r_{c}$. Possibly passing to a subsequence, we may assume that $h_l^{+}$ is strictly decreasing and there exists a sequence of real numbers $t_l$, such that $h_l^{+}<t_l\le h_{l-1}^{+}$, $(X_l,B_l(t_l)+M_l(t_l))$ is g-lc, and $-(K_{X_{l}}+B_{l}(t_l)+M_{l}(t_l))$ is not pseudo-effective over $Z$.
	
	 We may run an MMP$/Z_{l}$ on $-(K_{X_l}+B_l(t_l)+M_l(t_l))$ which terminates with a model $Y_l\to Z_l'$ over $Z_{l}$, such that $-(K_{Y_l}+B_{Y_l}(t_l)+M_{Y_l}(t_l))$ is antiample over $Z_l'$, where $B_{Y_l}(t),M_{Y_l}(t)$ are the strict transforms of $B_l(t),M_l(t)$ on $Y_l$. Since $-(K_{Y_l}+B_{Y_l}(h_l^{+})+M_{Y_l}(h_l^{+}))$ is pseudo-effective over $Z_l$, $-(K_{Y_l}+B_{Y_l}(h_l^{+})+M_{Y_l}(h_l^{+}))$ is nef over $Z_l'$. Thus there exists a real number $\eta_l$ such that $h_l^{+}\le\eta_l<t_l$ and 
	$$(K_{Y_l}+B_{Y_l}(\eta_l)+M_{Y_{l}}(\eta_l))|_{F_l}\equiv 0,$$
	where $F_l$ is a general fiber of $Y_l\to Z_l'$. 
	Since $({X_l/Z_{l}},B_l(r_{c})+M_l(r_{c}))$ is $\Rr$-complementary, $(Y_l,B_{Y_l}(r_c)+M_{Y_l}(r_{c}))$ is g-lc. By Theorem \ref{thm: uniformperturbationofglc}, we may assume that both $({Y_l},B_{Y_l}(t_l)+M_{Y_l}(t_l))$ and $({Y_l},B_{Y_l}(h_l^{+})+M_{Y_l}(h_l^{+}))$ are g-lc. Thus $(K_{Y_l}+B_{Y_l}(\eta_l)+M_{Y_l}(\eta_l))|_{F_l}$ is g-lc. Since $\lim_{l\to+\infty} \eta_l=r_{c}$, by Theorem \ref{thm:accumulationpoints1}, $r_{c}\in\Span_{\Qq}(\{r_0,r_1,\ldots,r_{c-1}\})$, a contradiction.
	
	\medskip
	
	For any $t$ satisfying $|t-r_{c}|\leq\epsilon$, we may run an MMP$/Z$ on $-(K_X+B(t)+M(t))$ which terminates with a model $Y_t$, such that $-(K_{Y_t}+B_{Y_t}(t)+M_{Y_t}(t))$ is semiample over $Z$, where $B_{Y_t}(t),M_{Y_t}(t)$ are the strict transforms of $B(t),M(t)$ on $Y_t$. Since $(Y_t,B_{Y_t}(r_c)+M_{Y_t}(r_c))$ is g-lc, by the claim again, $(Y_t,B_{Y_t}(t)+M_{Y_t}(t))$ is g-lc. Thus $(Y_t/Z,B_{Y_t}(t)+M_{Y_t}(t))$ is $\Rr$-complementary and so is $(X/Z,B(t)+M(t))$. This complete the proof.
\end{proof}

By Theorem \ref{thm: uniformperturbationofRcomp} and \cite[Theorem 5.17]{HLS19}, we can show the existence of Han type polytopes for $\Rr$-complementary (generalized) pairs. The proof is very similar to Theorem \ref{thm: uniformperturbationofRcomp} and \cite[Theorem 5.17]{HLS19}, we may omit it.

\begin{thm}[Han type polytopes for $\Rr$-complementary (generalized) pairs]\label{thm: uniformpolytopeforRcomp}
	Let $d,m$ and $l$ be positive integers, $\bm{v}_0:=(v_1^0,\ldots,v_{m+l}^0)\in\Rr^{m+l}$ a point and $V\subseteq\Rr^{m+l}$ the rational envelope of $\bm{v}_0$. Then there exists an open set $U\ni\bm{v}_0$ of $V$ depending only on $d,m,l$ and $\bm{v}_0$ satisfying the following.   
	
	Assume that $(X,(\sum_{i=1}^m v_i^0B_{i})+(\sum_{j=1}^{l}v_{m+j}^0M_j))$ is a generalized pair with data $X'\xto X\to Z$ and $\sum_{j=1}^{l}v_{m+j}^0M_j'$ such that
	\begin{enumerate} 
	   \item $\dim X=d$, 
	   \item $X$ is of Fano type over $Z$,
	   \item $B_1,\ldots,B_m\ge0$ are Weil divisors on $X$,
	   \item $M_j'$ is b-Cartier nef$/Z$ for any $1\le j\le l$, and
	   \item $(X/Z,(\sum_{i=1}^m v_i^0B_{i})+(\sum_{j=1}^{l}v_{m+j}^0M_j))$ is $\Rr$-complementary.
	\end{enumerate}
	Then $(X/Z,(\sum_{i=1}^m v_iB)+(\sum_{j=1}^{l}v_{m+j}M_j))$ is $\Rr$-complementary for any point $(v_1,\dots,v_{m+l})\in U$.
\end{thm}

%\section{Existence of $(n,\Ii_0)$-decomposable $\Rr$-complements}
\section{Boundedness of relative complements}\label{section4'}
 
 In this section, we show the following result which is the relative version of \cite[Theorem 1.10]{Bir19}. Proofs are very similar to those in \cite[Section 8]{Bir19}.
\begin{thm}\label{thm: relncomplforfiniterat}
     	Let $d$ and $r$ be positive integers and $\Ii\subseteq [0,1]\cap\Qq$ a finite set. Then there exists a positive integer $n$ depending only on $d$, $r$ and $\Ii$ satisfying the following. 
	
	Assume that $(X,B+M)$ is a generalized pair with data $X'\xto X\to Z$ and $M'$ such that
	\begin{itemize}
		\item $(X,B+M)$ is a g-lc pair of dimension $d$,	
		\item $B\in\Phi(\Ii)$ and $rM'$ is b-Cartier nef$/Z$,
		\item $X$ is of Fano type over $Z$, and
		\item $-(K_X+B+M)$ is nef over $Z$.
	\end{itemize}
	Then for any point $z\in Z$, $(X/Z\ni z,B+M)$ has a monotonic $n$-complement $(X/Z\ni z,B^++M)$.
\end{thm}

\begin{prop}\label{prop: relcomplspecalcase}
     Let $d\ge2$ be a positive integer. Then $($Theorem \ref{thm: relncomplforfiniterat}$)_{d-1}$ implies $($Theorem \ref{thm: relncomplforfiniterat}$)_{d}$ for those $(X,B+M)$ such that
     \begin{enumerate}
        \item $B\in\Ii$,
        \item $(X,\Sigma+\alpha M)$ is $\Qq$-factorial generalized plt for some boundary $\Sigma$ and $\alpha\in(0,1),$
        \item $-(K_X+\Sigma+\alpha M)$ is ample over $Z$,
        \item $S=\lf \Sigma\rf\le\lf B\rf$ is irreducible, and
        \item $S$ intersects $\pi^{-1}(z)$, where $\pi$ is the morphism $X\to Z$.
     \end{enumerate}
\end{prop}

\begin{proof}
     We may find a boundary $\Sigma_1$ such that $(X,\Sigma_1)$ is plt, $-(K_X+\Sigma_1)$ is ample over $Z$, and $\lf \Sigma_1\rf=S$. By the same arguements as in \cite[Proposition 8.1]{Bir19}, $S\to \pi(S)$ is a contraction. 
     
     Possibly replacing $X'$ by a higher model, we may assume that $f:X'\to X$ is a log resolution of $(X,B+\Sigma)$, $rM'$ is Cartier, and the induced map $\phi:S'\to S$ is a morphism, where $S'$ is the strict transform of $S$ on $X'$. Let 
     $$N':=K_{X'}+B'+M'=f^*(K_X+B+M),$$ and 
     $$K_{X'}+\Sigma'+\alpha M':=f^*(K_X+\Sigma+\alpha M).$$ 
     Then we have the generalized adjunction
     $$K_{S}+B_S+M_S\sim_{\Qq}(K_X+B+M)|_S$$
     such that $rM_{S'}$ is Cartier. Moreover, by \cite[3.1(2)]{Bir19}, we may assume that
     $$r(K_{S}+B_S+M_S)\sim r(K_X+B+M)|_S.$$    
     
     By \cite[Lemma 3.3]{Bir19}, $B_S\in\Phi(\Ii_1)$ for some finite set $\Ii_1\subseteq[0,1]\cap\Qq$ depending only on $r$ and $\Ii$. Restricting $K_X+\Sigma+\alpha M$ to $S$ shows that $S$ is a of Fano type over $\pi(S)$ by \cite[2.13(6)]{Bir19}. Thus there exists a positive integer $n$ divisible by $r$ depending only on $d-1,r$ and $\Ii_1$ such that $(S,B_S+M_S)$ has a monotonic $n$-complement $(S,B_{S}^++M_S)$ over $z$, where $B_{S}^+:=B_S+R_S$ for some $\Rr$-divisor $R_S\ge0$. Possibly replacing $n$ by a larger number, we may assume that $n\Ii\subseteq\Zz$. Let $R_{S'}:=\phi^*R_S$, then we have
     \begin{align*}
       nN'|_{S'}\sim -n\phi^*(K_S+B_S+M_S)\sim nR_{S'}\ge0.
     \end{align*}
     In the following, we want to lift $R_{S'}$ from $S'$ to $X'$.
     
     Let $T':=\lf B'^{\ge0}\rf$ and $\Delta':=B'-T'$. Define
     $$L':=-nK_{X'}-nT'-\lf (n+1)\Delta'\rf-nM'=n\Delta'-\lf (n+1)\Delta'\rf+nN',$$
     which is an integral divisor. %Note that $L'=n\Delta'-\lf (n+1)\Delta'\rf+nN'$. 
     Possibly replacing $\Sigma'$ by $(1-a)\Sigma'+aB'$, and $\alpha M'$ by $((1-a)\alpha+a)M'$ for some $a\in(0,1)$ sufficiently close to $1$, we may assume that $\alpha$ is sufficiently close to $1$ and $B'-\Sigma'$ has sufficiently small coefficients.
     
     We claim that there exists a divisor $P'$ on $X'$, such that $P'$ is exceptional over $X$, $(X',\Lambda')$ is plt and $\lf \Lambda'\rf=S'$, where $\Lambda':=\Sigma'+n\Delta'-\lf (n+1)\Delta'\rf+P'.$ Indeed let $\mult_{S'}P'=0$, and for each prime divisor $D'\neq S',$ let
     $$\mult_{D'}P'=-\mult_{D'}\lf\Sigma'+n\Delta'-\lf(n+1)\Delta'\rf \rf=-\mult_{D'}\lf\Sigma'-\Delta'+\{(n+1)\Delta'\} \rf.$$
     This implies that $0\le \mult_{D'}P'\le 1$ for any prime divisor $D'$. In fact if $D'\neq S'$ is a component of $T'$, then $D'$ is not a component of $\Delta'$ and $\mult_{D'} \Sigma'\in(0,1)$, hence $\mult_{D'}P'=0.$ If $D'$ is not a component of $T'$, then $\mult_{D'}(\Sigma'-\Delta')=\mult_{D'}(\Sigma'-B')$ is sufficiently small and $\mult_{D'}P'=0.$
     
     It suffices to show that $P'$ is exceptional over $X$. Assume that $D'$ is a component of $P'$ that is not exceptional over $X$. Then $D'\neq S'$, and since $nB$ is integral, $\mult_{D'}n\Delta'$ is integral. Hence $\mult_{D'}\lf(n+1)\Delta'\rf=\mult_{D'}\lf n\Delta'\rf$ which implies that $\mult_{D'} P'=-\mult_{D'}\lf\Sigma'\rf=0$, a contradiction.    
       
     By construction,
     \begin{align*}
        (L'+P')|_{S'}&=(n\Delta'-\lf(n+1)\Delta'\rf+nN'+P')|_{S'}
                     \\&\sim_Z nR_{S'}+n\Delta_{S'}-\lf(n+1)\Delta_{S'}\rf+P_{S'}=:G_{S'},     
     \end{align*}
     where $\Delta_{S'}:=\Delta'|_{S'}$ and $P_{S'}:=P'|_{S'}$. Note that $G_{S'}$ is integral by the choice of $R_{S'}$. Moreover, as the coefficients of $n\Delta_{S'}-\lf(n+1)\Delta_{S'}\rf$ belong to $(-1,1)$, $G_{S'}\ge0$.
     
     Let $A:=-(K_X+\Sigma+\alpha M),$ and $A':=f^*A$. Then
     \begin{align*}
        L'+P'&=n\Delta'-\lf(n+1)\Delta'\rf+nN'+P'
            \\&=K_{X'}+\Sigma'+\alpha M'+A'+n\Delta'-\lf(n+1)\Delta'\rf+nN'+P'
            \\&=K_{X'}+\Lambda'+A'+\alpha M'+nN'.
     \end{align*}
     Possibly shrinking $Z$ near $z$, we may assume that $Z$ is affine. Since $A'+\alpha M'+nN'$ is nef and big over $Z$, and $(X',\Lambda'-S')$ is klt, $h^1(L'+P'-S)=0$ by the relative Kawamata-Viehweg vanishing theorem \cite[Theorem 1-2-5]{KMM85}, and hence
     $$H^0(L'+P')\to H^0((L'+P')|_{S'})$$
     is surjective. Therefore, there exists $G'\ge0$ on $X'$ such that $L'+P'\sim G'$ and $G'|_{S'}=G_{S'}.$
     
     As $P'$ is exceptional over $X$, we have
     $$f_*(L'+P')=L=-nK_X-nT-\lf(n+1)\Delta\rf-nM\sim_Z G\ge0,$$
     where $T,\Delta,L,G$ are the strict transforms of $T',\Delta',L',G'$ respectively. Since $nB$ is integral, $\lf(n+1)\Delta\rf=n\Delta,$ and
     $$-n(K_X+B+M)=-nK_X-nT-n\Delta-nM=L\sim_Z G=:nR\ge0.$$
     Let $B^+:=B+R$. Then we have $n(K_X+B^++M)\sim_{Z}0$.
     
     It is enough to show that $(X,B^++M)$ is g-lc over some neighborhood of $z$ since then $(X/Z\ni z,B^++M)$ is a monotonic $n$-complement of $(X/Z\ni z,B+M)$. 
     
     We first show that $R|_S=R_S$. Since
     $$nR':=G'-P'+\lf(n+1)\Delta'\rf-n\Delta'\sim L'+\lf(n+1)\Delta'\rf-n\Delta'=nN'\sim_{\Qq,X}0,$$
     and $\lf(n+1)\Delta\rf=n\Delta$, we have that $f_*(nR')=G=nR$ and $R'=f^*R$. Thus
     \begin{align*}
     nR_{S'}&=G_{S'}-P_{S'}+\lf(n+1)\Delta_{S'}\rf-n\Delta_{S'}\\&=(G'-P'+\lf(n+1)\Delta'\rf-n\Delta')|_{S'}=nR'|_{S'},
     \end{align*}
     which means that $R_{S'}=R'|_{S'}$. In particular, $R_S=R|_S$ and
     $$K_S+B_S^++M_S=K_S+B_S+R_S+M_S= (K_X+B^++M)|_S.$$
     %which give the generalized adjunction$$(K_S+B_S^++M_S)\sim_{\Qq} (K_X+B^++M)|_S.$$
     By the inversion of generalized adjunction \cite[Lemma 3.2]{Bir19}, $(X,B^++M)$ is g-lc near $S$. 
     
     Suppose that $(X,B^++M)$ is not g-lc. Then there exists a real number $a\in(0,1)$ which is sufficiently close to $1$, such that $(X, (aB^++(1-a)\Sigma)+((1-a)\alpha+a)M)$ is not g-lc and is g-lc near $S$. In particular, the generalized non-klt locus of $(X,(aB^++(1-a)\Sigma)+((1-a)\alpha+a)M)$ is not connected. Since
     $-(K_X+(aB^++(1-a)\Sigma)+((1-a)\alpha+a)M)=-a(K_X+B^++M)-(1-a)(K_X+\Sigma+\alpha M)$
     is ample over $Z$, it contracts the connectedness principle \cite[Lemma 2.14]{Bir19}. Therefore $(X,B^++M)$ is g-lc over a neighborhood of $z$.
\end{proof}

     %Now we can show Theorem \ref{thm: relncomplforfiniterat}.
\begin{proof}[Proof of Theorem \ref{thm: relncomplforfiniterat}]
     We show the statement by induction on the dimension. Assume that Theorem \ref{thm: relncomplforfiniterat} holds in dimension $d-1$.
          
     We may assume that $1\in\Ii$. According to Theorem \ref{thm: dcc limit glc divisor}, we may assume that $B\in\Ii$ and $(X,B+M)$ is $\Rr$-complementary. 
     
     Let $N\ge0$ be a Cartier divisor on $Z$ passing through $z$, and $t$ the g-lc threshold of $\pi^*N$ with respect to $(X,B+M)$ over a neighborhood of $z$, where $\pi$ is the morphism $X\to Z$. Let $\Omega_0:=B+t\pi^*N$. Possibly shrinking $Z$ near $z$, we may assume that $(X,\Omega_0+M)$ is g-lc. Let $(X'',\Omega_0''+M'')$ be a g-dlt modification of $(X,\Omega_0+M)$. Then $X''$ is of Fano type over $Z$. Moreover, there exists a boundary $\Omega_1''\le\Omega_0''$ such that $\Omega_1''\in\Ii$, some component of $\lf \Omega_1''\rf$ intersecting $\pi^{-1}(z)$, and $B\le \Omega_1,$ where $\Omega_1$ is the strict transform of $\Omega_1''$ on $\Omega_1$.
     We may run an MMP$/Z$ on $-(K_{X''}+\Omega_1''+M'')$ which terminates with a model $X'''$ such that $-(K_{X'''}+\Omega_1'''+M''')$ is nef over $Z$, where $\Omega_1''',M'''$ are the strict transforms of $\Omega_1'',M''$ respectively. Since $(X'',\Omega_0''+M'')$ is $\Rr$-complementary, $(X''',\Omega_0'''+M''')$ is g-lc and so is $(X''',\Omega_1'''+M''')$. Moreover, no component of $\lf \Omega_1''\rf$ is contracted by the MMP, as $a(S'',X'',\Omega_1''+M'')<0$ for any contracted divisor $S''$. Possibly replacing $(X,B+M)$ by $(X''',\Omega_1'''+M''')$, we may assume that $X$ is $\Qq$-factorial, $-(K_X+B+M)$ is nef over $Z$, and $\Supp\lf B\rf$ intersecting $\pi^{-1}(z)$.
     
     We claim that there exist boundaries $\tilde{\Delta}\le\Delta$ such that $-(K_X+\Delta+\alpha M)$ and $-(K_X+\tilde{\Delta}+\alpha M)$ are nef and big over $Z$, some component of $\lf\Delta\rf$ intersects $\pi^{-1}(z)$, $(X,\Delta+\alpha M)$ is g-dlt, and $(X,\tilde{\Delta}+\alpha M)$ is g-klt for some $\alpha\in(0,1)$.
     
     Since $-K_X$ is big over $Z$ and $-(K_X+B+M)$ is nef over $Z$,
     $$-(K_X+\alpha B+\alpha M)=-\alpha(K_X+B+M)-(1-\alpha)K_X$$
     is big over $Z$ for any $\alpha\in(0,1)$. Assume that $\alpha$ is sufficiently close to $1$, we define $\Delta$ as follows. For any prime divisor $D$ which is vertical over $Z$, we let $\mult_D\Delta=\mult_DB$, otherwise let $\mult_D\Delta=\mult_D\alpha B$. Then $(X,\Delta+\alpha M)$ is g-lc, $\alpha B\le \Delta\le B$, $\Supp\lf \Delta\rf$ intersects $\pi^{-1}(z)$, and $-(K_X+\Delta+\alpha M)$ is big over $Z$ as $\Delta=\alpha B$ near the generic fiber.
     
     Let $X\to V/Z$ be the contraction defined by $-(K_X+B+M)$. We may run an MMP$/V$ on $-(K_X+\Delta+\alpha M)$ which terminates with a model $X_1''$ such that $-(K_{X_1''}+\Delta_1''+\alpha M_1'')$ is nef over $V$, where $\Delta_1'',M_1''$ are the strict transforms of $\Delta,M$ respectively. 
     %Note that $-(K_{X''}+\Delta''+\alpha M'')$ is big over $Z$ and $-(K_{X''}+B''+\alpha M'')$ is the pullback of some ample$/Z$ divisor on $V$, where $B''$ is the strict transform of $Z$. 
     Possibly replacing $\Delta$ by $aB+(1-a)\Delta$ for some $a\in(0,1)$ sufficiently close to $1$, we may assume that $-(K_{X_1''}+\Delta_1''+\alpha M_1'')$ is nef and big over $Z$. Possibly replacing $(X,\Delta+\alpha M)$ by $(X_1'',\Delta_1''+\alpha M_1'')$ and $B$ by its strict transform, we may assume that $-(K_X+\Delta+\alpha M)$ is nef and big over $Z$. %Let $X\to T/Z$ be the morphism define by $-(K_X+\Delta+\alpha M)$.
     
     Let $X\to T/Z$ be the morphism define by $-(K_X+\Delta+\alpha M)$, $\tilde{\Delta}:=\beta\Delta$ for some $\beta\in(0,1)$. We may run an MMP$/T$ on $-(K_X+\tilde{\Delta}+\alpha M)$ and terminates with $X_2''$. Note that the MMP is $(K_X+B+M)$-trivial, as it is $(K_X+\Delta+\alpha M)$-trivial and $\Delta\le B$. Possibly replacing $X$ by $X_2''$, and also the corresponding divisors, and pick $\beta$ sufficiently close to $1$, we may assume that $-(K_X+\tilde{\Delta}+\alpha M)$ is nef and big over $Z$. Possibly replacing $(X,B+M)$ by a g-dlt modification, increasing $\alpha,\beta$, and replacing $(X,\Delta+\alpha M)$ and $(X,\tilde{\Delta}+\alpha M)$ by its crepant pullbacks, the claim holds. Moreover, possibly shrinking $Z$ near $z$, we may assume that every component of $\lf\Delta\rf$ intersects $\pi^{-1}(z)$.
     %we may assume that $-(K_X+\Delta+\alpha M)$ and $-(K_X+\tilde{\Delta}+\alpha M)$ are nef and big over $Z$, some component of $\lf\Delta\rf$ intersects $\pi^{-1}(z)$, $(X,\Delta+\alpha M)$ is g-lc, and $(X,\tilde{\Delta}+\alpha M)$ is g-klt.
     %\noindent\textbf{Step 3.} In this step we introduce $\Rr$-divisors $A$ and $G$. 
     
     We may write $-(K_X+\Delta+\alpha M)\sim_{\Rr,Z}A+G$, where $A\ge0$ is ample$/Z$ and $G\ge0$. Suppose that $\Supp G$ does not contain any generalized non-klt center of $(X,\Delta+\alpha M)$. Then $(X,(\Delta+\delta G)+\alpha M)$ is g-dlt for some sufficiently small positive real number $\delta$. Moreover,
     $$-(K_X+\Delta+\delta G+\alpha M)\sim_{\Rr,Z}\delta A+(1-\delta)(A+G)$$
     is ample over $Z$. There exists a boundary $\Sigma$ such that $(X,\Sigma+\alpha M)$ is generalized plt, $S:=\lf\Sigma\rf\subseteq\lf B\rf$ is irreducible and intersects $\pi^{-1}(z)$, and $-(K_X+\Sigma+\alpha M)$ is ample over $Z$. By Proposition \ref{prop: relcomplspecalcase}, the theorem holds. In the following, we may assume that $\Supp G$ contains some generalized non-klt center of $(X,\Delta+\alpha M)$.   
     %\noindent\textbf{Step 4.} In this step we define a boundary $\Omega$ such that $(X,\Omega+\alpha M)$ is g-lc, and $-(K_X+\Omega+\alpha M)$ is ample over $Z$. 
     
     Possibly replacing $\tilde{\Delta}$, we may assume that $\Delta-\tilde{\Delta}$ has sufficiently small coefficients, and the g-lc threshold $t$ of $G+\Delta-\tilde{\Delta}$ with respect to $(X,\tilde{\Delta}+\alpha M)$ over a neighborhood of $z$ is sufficiently small such that any generalized non-klt center of $(X,\Omega+\alpha M)$ is a generalized non-klt center of $(X,\Delta+\alpha M)$, where $\Omega:=\tilde{\Delta}+t(G+\Delta-\tilde{\Delta})$. Moreover,
     \begin{align*}
          &\ \ \ -(K_X+\Omega+\alpha M)=-(K_X+\tilde{\Delta}+t(G+\Delta-\tilde{\Delta})+\alpha M)
          \\&=-(K_X+\Delta+\alpha M)+\Delta-\tilde{\Delta}-t(G+\Delta-\tilde{\Delta})
          \\&\sim_{\Rr,Z}A+G-tG+(1-t)(\Delta-\tilde{\Delta})
          =tA+(1-t)(A+G+\Delta-\tilde{\Delta})
          \\&\sim_{\Rr,Z}tA-(1-t)(K_X+\tilde{\Delta}+\alpha M)
     \end{align*}
     is ample over $Z$.
     
     If $\lf\Omega\rf\neq0$, then there exist a component $S$ of $\lf\Omega\rf\subseteq\lf\Delta\rf\subseteq\lf B\rf$ and a boundary $\Sigma_0$ such that $(X,\Sigma_0+\alpha M)$ is generalized plt, $S=\lf\Sigma_0\rf$ intersects $\pi^{-1}(z)$, and $-(K_X+\Sigma_0+\alpha M)$ is ample over $Z$. We are done by Proposition \ref{prop: relcomplspecalcase}.
     
     Suppose that $\lf\Omega\rf=0$. Let $(X_3'',\Omega_3''+\alpha M_3'')$ be a g-dlt modification of $(X,\Omega+\alpha M)$. Possibly shrinking $Z$ near $z$, we may assume that every component of $\lf\Omega_3''\rf$ intersects $\pi^{-1}(z)$. We may run a g-MMP$/Z$ on $K_{X_3''}+\lf\Omega_3''\rf+\alpha M_3$ which terminates with $X$ since $\lf\Omega_3''\rf$ is the reduced exceptional divisor of $X_3''\to X$ and $(X,\alpha M)$ is g-klt. Let $X_3'''\to X$ be the last step of the g-MMP which is a divisorial contraction contracting a prime divisor $S_3'''$. Let $(X_3''',B_3'''+M_3''')$ be the crepant pullback of $(X,B+M)$, then $S_3'''$ is a component of $\lf B_3'''\rf$. We finish the proof by Proposition \ref{prop: relcomplspecalcase}.
     %Then $(X''',S'''+\alpha M''')$ is generalized plt and $-(K_{X'''}+S'''+\alpha M''')$ is ample over $Z$. Let $(X''',B'''+M'''),(X''',\Omega'''+\alpha M''')$ and $(X''',\Delta'''+\alpha M''')$ be the crepant pullbacks of $(X,B+M),(X,\Omega+\alpha M)$ and $(X,\Delta+\alpha M)$ respectively. Then $S'''$ is a component of $\lf B'''\rf$, $\lf\Omega'''\rf$ and $\lf\Delta'''\rf$, $-(K_{X'''}+\Omega'''+\alpha M''')$ is nef and big over $Z$, and $\Delta'''\le B'''$. Then there exists a boundary $\Sigma'''$ such that $(X''',\Sigma'''+\alpha M''')$ is generalized plt, $S'''=\lf\Sigma'''\rf$ intersecting $\pi^{-1}(z)$, and $-(K_{X'''}+\Sigma'''+\alpha M''')$ is ample over $Z$. Hence we finish the proof by Proposition \ref{prop: relcomplspecalcase}.
\end{proof}

\section{Proof of Theorem \ref{thm: (n,I)complforfiniterat}}\label{section4}

\subsection{From the DCC set to a finite set}

\begin{thm}\label{thm: dcc limit glc divisor}
     Let $d$ be a positive integer, $\alpha$ a positive real number, and $\Ii\subseteq [0,1]$ a DCC set. Then there exists a finite set $\Ii'\subseteq \bar\Ii$ and a projection $g:\bar\Ii\to \Ii'$ depending only on $d,\alpha$ and $\Ii$ satisfying the following. 

     Assume that $(X,(\sum_{i=1}^sb_iB_i)+(\sum_{j=1}^l\mu_jM_j))$ is a $\Qq$-factorial generalized pair with data $X'\xto X\rightarrow Z$ and $\sum_{j=1}^l\mu_jM_j'$ such that
     \begin{itemize}
       \item $(X,(\sum_{i=1}^sb_iB_i)+(\sum_{j=1}^l\mu_jM_j))$ is g-lc of dimension $d$,
      \item $B_i\ge0$ is a Weil divisor and $b_i\in\Ii$ for any $1\le i\le s$, and
      %\item $b_i\in\Ii$ for any $1\le i\le s$, and
      \item $\mu_j\in\Gamma$ and $M_j'$ is b-Cartier nef$/Z$ for any $1\le j\le l$.
     \end{itemize}
     Then
\begin{enumerate}
	\item  $\gamma+\alpha \ge g(\gamma)\ge \gamma$ for any $\gamma\in\Ii$, 
\item $g(\gamma')\ge g(\gamma)$ for any $\gamma'\ge \gamma$, $\gamma,\gamma'\in\Ii$, and
	\item  $(X,(\sum_{i=1}^s g(b_i)B_i)+(\sum_{j=1}^l g(\mu_j)M_j))$ is g-lc.
\end{enumerate}
\end{thm}
The proof of Theorem \ref{thm: dcc limit glc divisor} is very similar to \cite[Theorem 5.18]{HLS19}.
\begin{proof}
     Possibly replacing $\Ii$ by $\bar\Ii$, we may assume that $\Ii=\bar\Ii$. Let $\Ii'':=\overline{g\mathcal{LCT}(d,\Ii)}$ which is an ACC set by Theorem \ref{thm: ACCforglcts}. By \cite[Lemma 5.17]{HLS19}, there exists a finite set $\Ii'\subseteq\Ii$ and a projection $g:\Ii\to \Ii'$ such that
      \begin{itemize}
 	\item $\gamma+\alpha \ge g(\gamma)\ge \gamma$ for any $\gamma\in \Ii$, 
 	\item $g(\gamma')\ge g(\gamma)$ for any $\gamma'\ge \gamma$, $\gamma,\gamma'\in\Ii$, and
 	\item $\beta\ge g(\gamma)$ for any $\beta\in \Ii''$ and $\gamma\in \Ii$ with $\beta\ge \gamma$.
     \end{itemize}
     
     It suffices to show that $(X,(\sum_{i=1}^s g(b_i)B_i)+(\sum_{j=1}^l g(\mu_j)M_j))$ is g-lc. We first show that $(X,(\sum_{i=1}^s g(b_i)B_i)+(\sum_{j=1}^l \mu_jM_j))$ is g-lc. Otherwise, there exists $0\le s'\le s-1$, such that $(X,(\sum_{i=1}^{s'} g(b_i)B_i+\sum_{i=s'+1}^s b_iB_i)+(\sum_{j=1}^l \mu_jM_j))$ is g-lc, and $(X,(\sum_{i=1}^{s'+1} g(b_i)B_i+\sum_{i=s'+2}^s b_iB_i)+(\sum_{j=1}^l \mu_jM_j))$ is not g-lc. Let 
     $$\beta:=\lct(X,(\sum_{i=1}^{s'} g(b_i)B_i+\sum_{i=s'+2}^s b_iB_i)+(\sum_{j=1}^l\mu_jM_j);B_{s'+1}+0).$$
     Then $g(b_{s'+1})> \beta\ge b_{s'+1}$. Since $g(b_i),b_i\in\Ii$ for any $i$, $\beta\in \Ii''$ and $\beta\ge g(b_{s'+1})$, a contradiction. Thus $(X,(\sum_{i=1}^s g(b_i)B_i)+(\sum_{j=1}^l \mu_jM_j))$ is g-lc. Then by the same argument as above, the claim holds. 
\end{proof}

\begin{thm}\label{thm: dcc limit lc complementary}
	Let $d$ be a positive integer, and $\Ii\subseteq [0,1]$ a DCC set. Then there exists a finite set $\Ii'\subseteq \bar\Ii$, and a projection $g:\bar\Ii\to \Ii'$ depending only on $d$ and $\Ii$ satisfying the following. 
	
	Assume that $(X,(\sum b_iB_i)+(\sum\mu_jM_j))$ is a $\Qq$-factorial generalized pair with data $X'\xto X\to Z$ and $\sum\mu_jM_j'$ such that
	\begin{itemize}
	    \item $\dim X=d$, 
	    \item $X$ is of Fano type over $Z$,
	    \item $B_i\ge0$ is a Weil divisor and $b_i\in\Ii$ for any $i$,
	    \item $\mu_j\in\Gamma$ and $M_j'$ is b-Cartier nef$/Z$ for any $j$, and
	    \item $(X/Z,(\sum b_iB_i)+(\sum\mu_jM_j))$ is $\Rr$-complementary.
	\end{itemize}
	Then
	\begin{enumerate}
		\item  $g(\gamma)\ge \gamma$ for any $\gamma\in \Ii$,
	    \item $g(\gamma')\ge g(\gamma)$ for any $\gamma'\ge \gamma$, $\gamma,\gamma'\in\Ii$, and
		\item  $(X/Z,(\sum g(b_i)B_i)+(\sum g(\mu_j)M_j))$ is $\Rr$-complementary.
	\end{enumerate}
\end{thm}
The proof of Theorem \ref{thm: dcc limit lc complementary} is very similar to that of \cite[Theorem 5.20]{HLS19}.
\begin{proof}
	We first claim that there exist a finite set $\Ii'\subseteq \bar\Ii$, and a projection $g:\bar\Ii\to \Ii'$ depending only on $d$ and $\Ii$ such that for any generalized pair $(X,(\sum b_iB_i)+(\sum \mu_jM_j))$ with data $X'\xto X\to Z$ and $\sum \mu_jM_j'$ satisfying the conditions, then
     \begin{itemize}
		\item[(i)]  $g(\gamma)\ge \gamma$ for any $\gamma\in \Ii$,
	    \item[(ii)] $g(\gamma')\ge g(\gamma)$ for any $\gamma'\ge \gamma$, $\gamma,\gamma'\in\Ii$,
		\item[(iii)]  $(X,(\sum g(b_i)B_i)+(\sum g(\mu_j)M_j))$ is g-lc, and
		\item[(iv)] $-(K_X+\sum g(b_i)B_i+\sum g(\mu_j)M_j)$ is pseudo-effective over $Z$.
	\end{itemize}
	
	We may assume that $1\in\Ii$. We first show that there exists a finite set $\Gamma'\subseteq\bar{\Gamma}$ and a projection $g:\bar{\Gamma}\to\Gamma'$ satisfy (i), (ii), (iii) and
	\begin{itemize}
	  \item[(iv)'] $-(K_X+\sum g(b_i)B_i+\sum \mu_jM_j)$ is pseudo-effective over $Z$.
	\end{itemize}
	Suppose on the contrary and by Theorem \ref{thm: dcc limit glc divisor}, there exist a sequence of $d$-dimensional generalized pairs $(X_{k},B_{(k)}+M_{(k)})$ with data $X_k'\xrightarrow{f_{k}} X_k\to Z_{k}$ and $M_{(k)}'$, where $B_{(k)}:=\sum_{i} b_{k,i}B_{k,i},M_{(k)}:=\sum_j \mu_{k,j}M_{k,j}$, and a sequence of projections $g_{k}:\bar\Ii\to \bar{\Ii}$, such that for any $k,i,$ we have
	\begin{itemize}
	    \item $b_{k,i}\in\Ii,b_{k,i}+\frac{1}{k}\ge g_k(b_{k,i})\ge b_{k,i}$,
	    \item $(X_{k}/Z_k,B_{(k)}+M_{(k)})$ is $\Rr$-complementary
	    \item $(X_{k},B_{(k)}''+M_{(k)})$ is g-lc, where $B_{(k)}'':=\sum_{i} g_k(b_{k,i})B_{k,i}$, and
	    \item $-(K_{X_{k}}+B_{(k)}''+M_{(k)})$ is not pseudo-effective over $Z_k$. 
	\end{itemize}

	Possibly replacing $X_k$ by a dlt modification of $(X_{k},B_{(k)}+M_{(k)})$, we may assume that $X_k$ is $\Qq$-factorial for any $k$. We may run an MMP$/Z_{k}$ on $-(K_{X_{k}}+B_{(k)}''+M_{(k)})$ with scaling of an ample$/Z_{k}$ divisor which terminates with a Mori fiber space $Y_k\to Z_k'$ over $Z_k$, such that $-(K_{Y_k}+B_{(Y_{k})}''+M_{Y_k})$ is antiample over $Z_k'$, where $B_{(Y_{k})}'',M_{Y_k}$ are the strict transforms of $B_{(k)}'',M_{(k)}$ on $Y_k$. Since $-(K_{X_k}+B_{(k)}+M_{(k)})$ is pseudo-effective over $Z_k$, $-(K_{Y_k}+B_{(Y_{k})}+M_{Y_k})$ is nef over $Z_k'$, where $B_{(Y_{k})}$ is the strict transform of $B_{(k)}$ on $Y_k$. 
	
	For each $k$, there exist a positive integer $k_j$ and a positive real number $0\le b_k^{+}\le 1$, such that $b_{k,k_j}\le b_k^{+}< g_k(b_{k,k_j})$, and $K_{F_k}+B_{F_k}^{+}+M_{F_{k}}\equiv 0$, where
	\begin{align*}
	K_{F_k}+B_{F_k}^{+}+M_{F_k}:=&(K_{Y_{k}}+(\sum_{i<k_j}g_k(b_{k,i})B_{Y_k,i})+b_k^{+}B_{Y_k,k_j}\\&+(\sum_{i>k_j}b_{k,i}B_{Y_k,i})+M_{Y_k})|_{F_k},
	\end{align*}
	$B_{Y_k,i}$ is the strict transform of $B_{k,i}$ on $Y_k$ for any $i$, and $F_k$ is a general fiber of $Y_k\to Z_k'$. Since $(X_{k},B_{(k)}+M_{(k)})$ is $\Rr$-complementary, $(Y_{k},B_{(Y_{k})}+M_{(k)})$ is g-lc. Thus $(Y_{k},B_{(Y_{k})}'+M_{Y_k})$ and $(F_k,B_{F_k}^{+}+M_{F_k})$ are g-lc.
	
	Since $g_k(b_{k,k_j})$ belongs to the DCC set $\bar{\Ii}$ for any $k,k_j$, possibly passing to a subsequence, we may assume that $g_k(b_{k,k_j})$ is increasing. Since $g_k(b_{k,k_j})-b_k^{+}>0$ and 
	$\lim_{k\to +\infty} (g_k(b_{k,k_j})-b_k^{+})=0,$
	by \cite[Lemma 5.21]{HLS19}, possibly passing to a subsequence, we may assume that $b_{k}^{+}$ is strictly increasing. 
	
	Now $K_{F_k}+B_{F_k}^{+}+M_{F_k}\equiv0$, the coefficients of $B_{F_k}^{+}$ belong to the DCC set $\bar{\Ii}\cup\{b_k^{+}\}_{k=1}^{\infty}$, and $b_{k}^{+}$ is strictly increasing. This contradicts Theorem \ref{thm: globalACC}. Thus there exist a finite set $\Gamma'\subseteq\bar{\Gamma}$ and a projection $g:\bar{\Gamma}\to\Gamma'$ satisfy (i), (ii), (iii) and (iv). Following the same argument as above, there exists a finite set $\Gamma'\subseteq\bar{\Gamma}$ and a projection $g:\bar{\Gamma}\to\Gamma'$ satisfy (i), (ii), (iii) and
	\begin{itemize}
	  \item[(iv)''] $-(K_X+\sum g(b_i)B_i+\sum g(\mu_j)M_j)$ is pseudo-effective over $Z$.
	\end{itemize}
	
	\medskip

	Since $-(K_{X}+\sum g(b_i)B_{i}+\sum g(\mu_j)M_j)$ is pseudo-effective over $Z$, we may run an MMP$/Z$ on $-(K_{X}+\sum g(b_i)B_{i}+\sum g(\mu_j)M_j)$ which terminates with a good minimal model $X''$, such that $-(K_{X''}+\sum g(b_i)B_{i}''+\sum g(\mu_j)M_j'')$ is semiample over $Z$, where $B_{i}'',M_j''$ are the strict transforms of $B_i,M_j$ on $X''$ respectively. Since $(X/Z,(\sum b_iB_i)+(\sum g(\mu_j)M_j))$ is $\Rr$-complementary, $(X''/Z,(\sum b_iB_{i}'')+(\sum \mu_jM_j''))$ is $\Rr$-complementary. Hence $(X'',\sum b_iB_{i}''+(\sum \mu_jM_j''))$ is g-lc, and $(X'',\sum g(b_i)B_{i}''+(\sum g(\mu_j)M_j''))$ is g-lc. Thus $(X''/Z,(\sum g(b_i)B_{i}'')+(\sum g(\mu_j)M_j'')$ is $\Rr$-complementary, and $(X/Z,(\sum g(b_i)B_i)$ $+\sum g(\mu_j)M_j)$ is also $\Rr$-complementary.	
\end{proof}

\subsection{Proof of Theorem \ref{thm: (n,I)complforfiniterat}}
\begin{proof}[Proof of Theorem \ref{thm: (n,I)complforfiniterat}]   
     We may assume that $1\in\Ii$. Possibly replacing $(X,B+M)$ by a g-dlt modification, we may assume that $X$ is $\Qq$-factorial. By Theorem \ref{thm: dcc limit lc complementary}, we may assume that $\Ii$ is a finite set.
     
     By Theorem \ref{thm: uniformpolytopeforRcomp}, there exist two finite sets $\Ii_0\subseteq(0,1]$ and $\Ii_1\subseteq[0,1]\cap\Qq$ depending only on $d$ and $\Ii$ such that possibly shrinking $Z$ near $z$, we have $(X/Z,B_i+M_{(i)})$ is $\Rr$-complementary, $\sum a_i=1$, and
     $$K_X+B+M=\sum a_i(K_X+B_i+M_{(i)})$$
     for some $a_i\in\Ii_0,B_i\in\Ii_1$ and $M_{(i)}'=\sum_j \mu_{ij}M_{j}'$ with $\mu_{ij}\in\Gamma_1$. Moreover, if $\bar{\Ii}\subseteq\Qq$, then we may pick $\Ii_0=\{1\},$ and $B_i=B.$
     
     For each $i$, we may run an MMP$/Z$ on $-(K_X+B_i+M_{(i)})$ and terminates with a model $Y_i$, such that $-(K_{Y_i}+B_{Y_i,i}+M_{Y_i})$ is nef over $Z$, where $B_{Y_i,i}$ and $M_{Y_i}$ are the strict transforms of $B_i$ and $M'$ on $Y_i$ respectively. Since $(X/Z,B+M)$ is $\Rr$-complementary, $(Y_i,B_{Y_i}+M_{Y_i})$ is g-lc, where $B_{Y_i}$ is the strict transform of $B$ on $Y_i$. According to the construction of $\Ii_1$, $({Y_i},B_{Y_i,i}+M_{Y_i})$ is g-lc. By Theorem \ref{thm: relncomplforfiniterat}, there exists a positive integer $n$ depending on $d$ and $\Ii_1$, such that $(Y_i/Z,B_{Y_i,i}+M_{Y_i}))$ has a monotonic $n$-complement $(Y_i/Z,B_{Y_i,i}^++M_{Y_i}))$.  
     
     By Lemma \ref{lem: pullbackcomplements}, $(X/Z,B_{i}+M_{(i)})$ has a monotonic $n$-complement $(X/Z,(B_{i}+G_i)+M_{(i)})$ for some $G_{i}\ge0$. Let $B^{+}:=\sum a_i(B_{i}+G_i)$. Then $(X/Z,B_{i}^++M)$ is an $(n,\Gamma_0)$-complement of $(X/Z,B_{i}+M)$.
\end{proof}

\section{Existence of $n$-complements}\label{section5}
\subsection{Diophantine approximation}
The following lemma is similar to \cite[Lemma 6.6]{HLS19} and \cite[Lemma 6.1]{CH20}, but in a slightly different form.
\begin{lem}\label{lem:rational direction}
     Let $n_0,s$  be two positive integers, $\epsilon_0$ a positive real number, $\bm{v}_0\in\Rr^s\setminus\Qq^{s}$ a point, $V\subseteq\Rr^s$ the rational envelope of $\bm{v}_0$, $||.||$ a norm on $V$ and $\bm{e}\in V$ a non-zero vector. Let $\Ii\subseteq[0,1]$ and $\Ii_0\subseteq(0,1]$ be finite sets. Then there exist a positive integer $n_0|n$ and a vector $\bm{v}$ depending only on $n_0,s,\epsilon_0,\bm{v}_0,||.||,\bm{e},\Ii$ and $\Ii_0$ satisfying the following.

     Assume that $a_i\in\Ii_0$ and $b_{ij}\in\frac{1}{n_0}\Zz\cap[0,1]$ $(1\le i\le k,1\le j\le m)$ such that $\sum_{i=1}^k a_i=1$ and $\sum_{i=1}^k a_ib_{ij}\in\Ii$. Then there exists a point $\bm{a}':=(a_1',\dots,a_k')\in\Rr^k_{>0}$ such that
     \begin{enumerate}
       \item $\sum_{i=1}^k a_i'=1$,
       \item $n(\bm{a}',\bm{v})\in n_0\Zz^{s+k},$
       \item $||\bm{v}_0-\bm{v}||<\frac{\epsilon_0}{n}$,
       \item $||\frac{\bm{v}_0-\bm{v}}{||\bm{v}_0-\bm{v}||}-\frac{\bm{e}}{||\bm{e}||}||<\epsilon_0$, and
       \item $n\sum_{i=1}^k a_i'b_{ij}=n\lfloor \sum_{i=1}^k a_ib_{ij} \rfloor+\lfloor (n+1)\{\sum_{i=1}^k a_ib_{ij}\}\rfloor$ for any $1\le j\le m$.
     \end{enumerate}
     
     Moreover, if there exist real numbers $r_0:=1,r_1,\dots,r_c$ which are linearly independent over $\Qq$ such that $\Ii\subseteq\Span_{\Qq_{\ge0}}(\{r_0,\dots,r_c\}),$ then we can additionally require that $\sum_{i=1}^k a_i'b_{ij}\ge \sum_{i=1}^k a_ib_{ij}$ for any $1\le j\le m.$
\end{lem}

\begin{proof}  
     Let $\bm{v}_1,\dots,\bm{v}_c\in\Zz^s$ be a basis of $V$ such that $\bm{e}\in\Span_{\Rr_{\ge0}}(\{\bm{v}_1,\dots,\bm{v}_c\})$. Then there exist $c'\ge c,r_1,\dots,r_{c'}$ and $e_1,\dots,e_c\ge0$ such that $r_0:=1,r_1,\dots,r_{c'}$ are linearly independent over $\Qq$, $\Ii_0\subseteq\Span(\{r_0,\dots,r_{c'}\})$, and
     $$\bm{v}=\sum_{i=1}^c r_i\bm{v}_i,\bm{e}=\sum_{i=1}^c e_i\bm{v}_i.$$ 
     
    There exist positive integers $l,M$ and $\Qq$-linear functions $a_i(\bm{r}):\Rr^{c'}\to \Rr$ depending only on $\Ii_0$ such that $a_i(\bm{r}_0)=a_i,$ $la_i(\bm{r})$ is a $\Zz$-linear function and 
    $$|a_i(\bm{r})-a_i(\bm{0})|\le M||\bm{r}||_{\infty}$$ 
    for any $\bm{r}\in\Rr^{c'}$ and $i,$ where $\bm{r}_0:=(r_1,\dots,r_{c'})$ and $\bm{0}:=(0,\dots,0)\in\Rr^{c'}$.
    
    We define a norm $||.||_*$ on $V$. For any $\bm{x}\in V$, there exist unique real numbers $x_1,\dots,x_c$ such that $\bm{x}=\sum_{i=1}^c x_i\bm{v}_i$, we define $||\bm{x}||_*:=\max_{1\le i\le c}\{|x_i|\}$.
    
    By \cite[Lemma 6.4]{HLS19}, there exists a positive real number $M_1$ such that 
    $$||\frac{\bm{x}}{||\bm{x}||}-\frac{\bm{y}}{||\bm{y}||}||\le M_1||\frac{\bm{x}}{||\bm{x}||_*}-\frac{\bm{y}}{||\bm{y}||_*}||_*$$
    for any non-zero vectors $\bm{x},\bm{y}\in V$. Moreover, possibly replacing $M_1$ by a larger number, we may assume that $||\bm{x}||\le M_1||\bm{x}||_*$ for any vector $\bm{x}\in V$.
    
    Let $\epsilon'$ be a positive real number such that 
    $$\epsilon'<\min_{\gamma_1\in\Ii,\gamma_2\in\Ii_0}\{\gamma_1>0,1-\gamma_1>0,\gamma_2,1\},$$
    and $\epsilon''$ a positive real number such that $\epsilon''<\min\{\frac{\epsilon'^2}{M},\frac{\epsilon_0}{M_1}\}$. By \cite[Lemma 6.5]{HLS19}, there exist an integer $ln_0|n$ and a point $\bm{r}_0':=(r_1',\dots,r_{c'}')\in\Rr^{c'}$, such that
    \begin{itemize}
       %\item $ln_0|n$,
       \item $n\bm{r}_0'\in ln_0\Zz^{c'}$,
       \item $||\bm{r}_0'-\bm{r}_0||_{\infty}<\frac{\epsilon''}{n}$, and
       \item $||\frac{\bm{c}-\bm{d}}{||\bm{c}-\bm{d}||_\infty}
             -\frac{\bm{e}_0}{||\bm{e}_0||_\infty}||_\infty<\epsilon''$, where $\bm{c}:=(r_1',\dots,r_c'),\bm{d}:=(r_1,\dots,r_c)$ and $\bm{e}_0:=(e_1,\dots,e_c)$.
    \end{itemize}
    
    Let $\bm{v}:=\sum_{i=1}^c r_i'\bm{v}_i$, then we have $n\bm{v}=\sum_{i=1}^c nr_i'\bm{v}_i\in n_0\Zz^s$,
    $$||\bm{v}-\bm{v}_0||\le M_1||\bm{v}-\bm{v}_0||_*\le M_1||\bm{r}_0'-\bm{r}_0||_{\infty}<\frac{\epsilon_0}{n},$$
    and
    \begin{align*}
    &||\frac{\bm{v}-\bm{v}_0}{||\bm{v}-\bm{v}_0||}
             -\frac{\bm{e}}{||\bm{e}||}||\le M_1 ||\frac{\bm{v}-\bm{v}_0}{||\bm{v}-\bm{v}_0||_*}
             -\frac{\bm{e}}{||\bm{e}||_*}||_*\\&=M_1||\frac{\bm{c}-\bm{d}}{||\bm{c}-\bm{d}||_\infty}
             -\frac{\bm{e}_0}{||\bm{e}_0||_\infty}||_\infty
             <M_1\epsilon''<\epsilon_0.
     \end{align*}
    
    Let $a_i':=a_i(\bm{r}_0')$ for any $1\le i\le k$ and $\bm{a}':=(a_1',\dots,a_k')$. Since $\sum_{i=1}^k a_i=1$ and $r_0,r_1,\dots,r_{c'}$ are linearly independent over $\Qq,\sum_{i=1}^k a_i(\bm{r})=1$ for any $\bm{r}\in\Rr^{c'}$. In particular, $\sum_{i=1}^k a_i'=1.$ Moreover, we have 
    $ na_i'=n_0\frac{n}{ln_0}la_i(\bm{r}_0')\in n_0\Zz$ for any $1\le i\le k,$ and 
    $$||\bm{a}-\bm{a}'||_\infty\le M||\bm{r}_0-\bm{r}'_0||_\infty< M\frac{\epsilon''}{n}<\frac{\epsilon'^2}{n},$$
    where $\bm{a}:=(a_1,\dots,a_k)$. In particular, $a_i'$ are positive real numbers, since 
    $$a_i'\ge a_i-|a_i-a_i'|\ge a_i-||\bm{a}-\bm{a}'||_\infty>a_i-\frac{\epsilon'^2}{n}>0.$$ 
   
    It suffices to show $(6)$. If $\sum_{i=1}^k a_ib_{ij}=1$ for some $j$, then $b_{ij}=1$ for any $1\le i\le k$ as $\sum_{i=1}^k a_i=1$ and $1\ge b_{ij}\ge0.$ Thus $\sum_{i=1}^k a_i'b_{ij}=1$. Hence we may assume that $1>\sum_{i=1}^k a_ib_{ij}>0$ and $\sum_{i=1}^k a_i'b_{ij}<1$. Since $n\sum_{i=1}^k a_i'b_{ij}=\sum_{i=1}^k \frac{n}{n_0}a_i'\cdot(n_0b_{ij})\in\Zz$, we only need to show that
    $$n\sum_{i=1}^k a_i'b_{ij}+1>(n+1)\sum_{i=1}^k a_ib_{ij}\ge n\sum_{i=1}^k a_i'b_{ij}.$$    
    The above inequalities hold since $\sum_{i=1}^k a_ib_{ij}\in\Ii$, $k\epsilon'<\sum_{i=1}^k a_i=1$, and
    $$ n\sum_{i=1}^k |a_i'-a_i|b_{ij}< nk\cdot\frac{\epsilon'^2}{n}<\epsilon'<\min\{\sum_{i=1}^k a_ib_{ij},1-\sum_{i=1}^k a_ib_{ij}\}.$$
     
     If $\Ii\subseteq\Span_{\Qq_{\ge0}}(\{r_0,\dots,r_c\}),$ then by our choices of $a_i'$ and $\bm{r}_0'$, 
     $$\sum_{i=1}^ka_i'b_{ij}=\sum_{i=1}^ka_i(\bm{r}_0')b_{ij}\ge\sum_{i=1}^ka_i(\bm{r}_0)b_{ij}=\sum_{i=1}^ka_ib_{ij}$$ 
     for any $1\le j\le s$.
\end{proof}

\subsection{Proof of Theorem \ref{thm: existence of n complement for generalized pairs}}
     
\begin{proof}[Proof of Theorem \ref{thm: existence of n complement for generalized pairs}]
     According to Theorem \ref{thm: (n,I)complforfiniterat}, there exist a positive integer $n_0$ and a finite set $\Ii_0\subseteq(0,1]$ depending only on $d$ and $\Ii$, such that possibly shrinking $Z$ near $z$, $(X/Z,B+M)$ has an $(n_0,\Ii_0)$-decomposable $\Rr$-complement $(X/Z,\tilde{B}+\tilde{M})$. In particular, there exist $a_i\in\Ii_0$ and boundaries $\tilde{B}_i'$ with nef parts $\tilde{M}_i'$ such that $(X/Z,\tilde{B}_i+\tilde{M}_i)$ is an $n_0$-complement of itself for any $i$, and
     $$K_X+\tilde{B}+\tilde{M}=\sum a_i(K_X+\tilde{B}_i+\tilde{M}_i).$$
     By Lemma \ref{lem:rational direction}, there exist a positive integer $n$ divisible by $pn_0$ and a vector $\bm{v}\in V$ depending only on $\epsilon,p,n_0,\bm{v}_0,\bm{e}$ and $\Ii_0$ such that there exist positive rational numbers $a_i'$ with the following properties:
     \begin{itemize}
         \item $\sum a_i'=1$,
         \item $n\bm{v}\in\Zz^s$, and $na_i'\in n_0\Zz$ for any $i$,
         \item $||\bm{v}-\bm{v}_0||<\frac{\epsilon}{n}$, 
         \item $||\frac{\bm{v}-\bm{v}_0}{||\bm{v}-\bm{v}_0||}-\frac{\bm{e}}{||\bm{e}||}||<\epsilon$,
         \item $nB^+\ge n\lfloor \tilde{B}\rfloor+\lfloor (n+1)\{\tilde{B}\}\rfloor$, where $B^+:=\sum a_i' \tilde{B}_i'$, and
         \item $n\sum_{i} a_i'\mu_{ij}=n\lfloor \sum_{i} a_i\mu_{ij} \rfloor+\lfloor (n+1)\{\sum_{i} a_i\mu_{ij}\}\rfloor$ for any $j$, where $\tilde{M}_i'=\sum_j\mu_{ij}M_j'$.
     \end{itemize}
     Let $M^{+'}=\sum_i a_i'\tilde{M}_i'=\sum_j(\sum_i a_i'\mu_{ij})M_j'$, then
     \begin{align*}
     n(K_{{X}}+B^++M^+)&=n\sum a_i'(K_{X}+\tilde{B}_i+\tilde{M}_i)
     \\&=\sum \frac{a_i'n}{n_0}\cdot n_0(K_{X}+\tilde{B}_i+\tilde{M}_i)\sim_Z0.
     \end{align*}
     Hence $(X/Z,B^++M^+)$ is an $n$-complement of $(X/Z,B+M)$, since $\tilde{B}\ge B$ and $\tilde{\mu}_j\ge\mu_j$ for any $j$.

     Moreover, if $\Span_{\Qq_{\ge0}}(\bar{\Ii}\backslash\Qq)\cap (\Qq\backslash\{0\})=\emptyset$, then ${B}^+\ge B'\ge B$ and $\mu_j^+\ge\tilde{\mu_j}\ge\mu_j$ for any $j$ by Lemma \ref{lem:rational direction} and \cite[Lemma 6.3]{HLS19}.
\end{proof}

\bibliographystyle{alpha}

\begin{thebibliography}{BCHM10}

\bibitem[Bir04]{Bir04}
C. Birkar.
\newblock Boundedness of $\epsilon$-log canonical complements on surfaces,
  \url{https://www.dpmms.cam.ac.uk/~cb496/surfcomp.pdf}.
\newblock {\em preprint}, 2004.

\bibitem[Bir16]{Bir16b}
C. Birkar.
\newblock Singularities of linear systems and boundedness of {F}ano varieties.
\newblock {\em arXiv:1609.05543}, 2016.

\bibitem[Bir18]{Bir18}
C. Birkar.
\newblock Log {C}alabi-{Y}au fibrations.
\newblock {\em arXiv:1811.10709}, 2018.

\bibitem[Bir19]{Bir19}
C. Birkar.
\newblock Anti-pluricanonical systems on {F}ano varieties.
\newblock {\em Ann. of Math. (2)}, 190(2):345--463, 2019.

\bibitem[BCHM10]{BCHM10}
C. Birkar, P. Cascini, C. D. Hacon, and J. McKernan.
\newblock Existence of minimal models for varieties of log general type.
\newblock {\em J. Amer. Math. Soc.}, 23(2):405--468, 2010.

\bibitem[BLX19]{BLX19}
H. Blum, Y. Liu, and C. Xu.
\newblock Openness of {K}-semistability for {F}ano varieties.
\newblock {\em arXiv:1907.02408}, 2019.

\bibitem[BZ16]{BZ16}
C. Birkar and D. Zhang.
\newblock Effectivity of {I}itaka fibrations and pluricanonical systems of
  polarized pairs.
\newblock {\em Publ. Math. Inst. Hautes \'Etudes Sci.}, 123:283--331, 2016.

\bibitem[CH20]{CH20}
G. Chen and J. Han.
\newblock Boundedness of $(\epsilon,n)$-complements for surfaces.
\newblock {\em arXiv:2002.02246}, 2020.

\bibitem[Fil17]{Fil20}
S. Filipazzi.
\newblock On a generalized canonical bundle formula and generalized adjunction.
\newblock {\em to appear in Annali della Scuola Normale di Pisa - Classe di Scienze.}, 2018.

\bibitem[FM18]{FM18}
S. Filipazzi and J. Moraga.
\newblock Strong $(\delta,n)$-complements for semi-stable morphisms.
\newblock {\em arXiv:1810.01990}, 2018.

\bibitem[FMX19]{FMX19}
S. Filipazzi, J. Moraga, and Y. Xu.
\newblock Log canonical 3-fold complements.
\newblock {\em arXiv:1909.10098}, 2019.

\bibitem[HL20]{HL17}
J. Han and Z. Li.
\newblock On {F}ujita's conjecture for pseudo-effective thresholds.
\newblock {\em Math. Res. Lett.}, volume 27, Number 2, 377-396,  2020.

\bibitem[HL18]{HL18}
J. Han and Z. Li.
\newblock Weak {Z}ariski decompositions and log terminal models for generalized
  polarized pairs.
\newblock {\em arXiv:1806.01234}, 2018.

\bibitem[HLiu20]{HL20}
J. Han and W. Liu.
\newblock On numerical nonvanishing for generalized log canonical pairs.
\newblock {\em Documenta Mathematica}, 20:1-5, 2020.

\bibitem[HLiu19]{HL19}
J. Han and W. Liu.
\newblock On a generalized canonical bundle formula for generically finite
  morphisms.
\newblock {\em arXiv:1905.12542v1}, 2019.


\bibitem[HLQ17]{HLQ17}
J. Han, Z. Li, and L. Qi.
\newblock {ACC} for log canonical threshold polytopes.
\newblock {\em arXiv:1706.07628}, 2017.

\bibitem[HLS19]{HLS19}
J. Han, J. Liu, and V. V. Shokurov.
\newblock {ACC} for minimal log discrepancies of exceptional singularities.
\newblock {\em arXiv:1903.04338v1}, 2019.

\bibitem[HM06]{HM06}
C. D. Hacon and J. McKernan.
\newblock Boundedness of pluricanonical maps of varieties of general types.
\newblock {\em Invent. Math.}, 166:1--25, 2006.

\bibitem[HMX14]{HMX14}
C. D. Hacon, J. McKernan, and C. Xu.
\newblock A{CC} for log canonical thresholds.
\newblock {\em Ann. of Math. (2)}, 180(2):523--571, 2014.

\bibitem[KMM87]{KMM85}
Y. Kawamata, K. Matsuda, and K. Matsuki.
\newblock Introduction to the minimal model problem.
\newblock In {\em Algebraic geometry, {S}endai, 1985}, volume~10 of {\em Adv.
  Stud. Pure Math.}, pages 283--360. North-Holland, Amsterdam, 1987.

\bibitem[Liu18]{LJH18}
J. Liu.
\newblock Toward the equivalence of the {ACC} for $a$-log canonical thresholds
  and the {ACC} for minimal log discrepancies.
\newblock {\em arXiv:1809.04839}, 2018.

  \bibitem[LT19]{LT19}
V. Lazic and N. Tsakanikas.
\newblock On minimal models.
\newblock {\em arXiv:1905.05576v2}, 2019.

\bibitem[LP18]{LP18-1}
V. Lazic and T. Peternell.
\newblock On generalised abundance, I.
\newblock {\em to appear in Publ. Res. Inst. Math. Sci.}, 2018.

\bibitem[LP19]{LP18-2}
V. Lazic and T. Peternell.
\newblock On generalised abundance, II.
\newblock {\em to appear in Peking Math. J.}, 2018.

\bibitem[Nak16]{Nak16}
Y. Nakamura.
\newblock On minimal log discrepancies on varieties with fixed gorenstein
  index.
\newblock {\em Michigan Math. J.}, 65(1):165--187, 03 2016.

\bibitem[PS01]{PS01}
Y. G. Prokhorov and V. V. Shokurov.
\newblock The first fundamental theorem on complements: from global to local.
\newblock {\em Izv. Ross. Akad. Nauk Ser. Mat.}, 65(6):99--128, 2001.

\bibitem[PS09]{PS09}
Y. G. Prokhorov and V. V. Shokurov.
\newblock Towards the second main theorem on complements.
\newblock {\em J. Algebraic Geom.}, 18(1):151--199, 2009.

\bibitem[Sho92]{Sho92}
V. V. Shokurov.
\newblock Three-dimensional log perestroikas.
\newblock {\em Izv. Ross. Akad. Nauk Ser. Mat.}, 56(1):105--203, 1992.

\bibitem[Sho00]{Sho00}
V. V. Shokurov.
\newblock Complements on surfaces.
\newblock {\em J. Math. Sci. (New York)}, 102(2):3876--3932, 2000.
\newblock Algebraic geometry, 10.

\bibitem[Sho20]{Sho19}
V. V. Shokurov.
\newblock Existence and boundedness of n-complements.
\newblock {\em preprint}, 2020.

\bibitem[XuC19]{Xu19}
C. Xu.
\newblock A minimizing valuation is quasi-monomial.
\newblock {\em  Annals of Math.}, 191:1003-1030,2020.  

\bibitem[XuY19]{Xuyanning19-1}
Y. Xu.
\newblock Complements on log canonical {F}ano varieties.
\newblock {\em arXiv:1901.03891}, 2019.

\end{thebibliography}

\end{document}